\documentclass[a4,12pt]{article} 

\usepackage{amssymb}
\usepackage{amsmath}
\usepackage{mathtools}
\usepackage{amsthm}
\usepackage{ascmac}

\usepackage{enumerate}

\usepackage{amsfonts}
\usepackage{mathrsfs}
\usepackage{bm}
\usepackage{bbm}

\usepackage[dvipdfmx, svgnames, dvipsnames, x11names]{xcolor} 

\usepackage[]{hyperref}
\hypersetup{
 setpagesize=false,
 hypertexnames=false,
 hyperfootnotes=true,
 bookmarksnumbered=true,%
 bookmarksopen=true,%
 colorlinks=true,%
 linkcolor=blue,
 citecolor=blue,
}

\numberwithin{equation}{section}

\theoremstyle{plain}
	\newtheorem{theorem}{Theorem}[section]
	\newtheorem{lemma}[theorem]{Lemma}
	\newtheorem{proposition}[theorem]{Proposition}
	\newtheorem{corollary}[theorem]{Corollary}
	
	\newtheorem{lem}[theorem]{Lemma}
	\newtheorem{prop}[theorem]{Proposition}

\theoremstyle{definition}
	\newtheorem{definition}[theorem]{Definition}
	\newtheorem{remark}[theorem]{Remark}
	\newtheorem*{acknowledgements}{Acknowledgements}
	\newtheorem{exam}{Example}
\theoremstyle{remark}
	
\newcommand{\C}{\mathbb{C}}
\newcommand{\N}{\mathbb{N}}

\newcommand{\R}{\mathbb{R}}

\newcommand{\boB}{\mathcal{B}}

\newcommand{\boD}{\mathcal{D}}

\newcommand{\boG}{\mathcal{G}}

\newcommand{\boK}{\mathcal{K}}

\newcommand{\boN}{\mathcal{N}}

\newcommand{\boR}{\mathcal{R}}

\newcommand{\boT}{\mathcal{T}}

\newcommand{\la}{\ensuremath{\lambda}}
\newcommand{\La}{\ensuremath{\Lambda}}

\newcommand{\al}{\ensuremath{\alpha}}

\newcommand{\si}{\ensuremath{\sigma}}
\newcommand{\om}{\ensuremath{\omega}}

\newcommand{\sech}{{\rm sech}}

\DeclareMathOperator{\supp}{{\rm supp}}
\DeclareMathOperator{\odd}{{\rm odd}}
\DeclareMathOperator{\even}{{\rm even}}

\newcommand{\Dgg}{{\Del^\gam_\boG}}
\newcommand{\Dgr}{{\Del_\R^\gam}}
\newcommand{\bvph}{{\bm{\vph}}}

\newcommand{\nor}[2]{\left\| {#1} \right\|_{#2}}
\newcommand{\jb}[1]{\langle {#1} \rangle{} }

\newcommand{\Del}{{\Delta}}
\newcommand{\del}{{\delta}}
\newcommand{\rd}{{\partial}}

\newcommand{\inp}[2]{\left\langle {#1} , {#2} \right\rangle }

\newcommand{\gam}{{\gamma}}

\newcommand{\vph}{{\varphi}}

\newcommand{\bcp}{\mathbbm{1}_{>0}}

\newcommand{\bof}{\bm{f}}
\newcommand{\bou}{\bm{u}}
\newcommand{\bow}{\bm{w}}
\newcommand{\bA}{\bm{A}}
\newcommand{\bB}{\bm{B}}
\newcommand{\bE}{\bm{E}}
\newcommand{\bF}{\bm{F}}
\newcommand{\bG}{\bm{G}}
\newcommand{\bU}{\bm{U}}
\newcommand{\bV}{\bm{V}}
\newcommand{\bW}{\bm{W}}

\newcommand{\li}{{\mathrm{line}}}

\begin{document}

\title{Global dynamics below a threshold for the nonlinear Schr\"{o}dinger equations with the Kirchhoff boundary and the repulsive Dirac delta boundary on a star graph}
\author{Masaru Hamano, Masahiro Ikeda, Takahisa Inui, Ikkei Shimizu}
\date{}

\maketitle

\begin{abstract}
We consider the nonlinear Schr\"{o}dinger equations on the star graph with the Kirchhoff boundary and the repulsive Dirac delta boundary at the origin. In the present paper, we show the scattering-blowup dichotomy result below the mass-energy of the ground state on the real line. The proof of the scattering part is based on a concentration compactness and rigidity argument. Our main contribution is to give a linear profile decomposition on the star graph by using a symmetrical decomposition. 
\end{abstract}

\tableofcontents


\section{Introduction}\label{Sec:Int}


\subsection{Background and Motivation}

The purpose of this paper is to analyze the long-time behavior of nonlinear Schr\"odinger equation on a star graph. 
A star graph with $N$ edges, denoted by $\boG$, 
is a metric graph made of $N$ half-lines whose origins are joined at one vertex. 
We identify $\boG$ with a union of $N$ copies of half-line $[0,\infty)$, 
and formulate functions on $\boG$ by vector-valued function on $[0,\infty)$. 
Then the nonlinear Schr\"odinger equation on $\mathcal{G}$ is 
\begin{align*}
	i\partial_t \bm{u} + \Delta_\mathcal{G}^\gamma \bm{u}
		=\mathcal{N}_{\mu}(\bm{u}),\quad(t,x) \in \mathbb{R}\times(0,\infty), 
\tag{NLS$_\mathcal{G}$}
\label{NLS}
\end{align*}
where $p > 1$, $\gamma \geq 0$, $N \geq 3$, $\bm{u}(t,x) = (u_k(t,x))_{k=1}^N:\mathbb{R}\times (0,\infty) \to \mathbb{C}^N$, $\mathcal{N}_{\mu}(\bm{u}) := (\mu |u_k|^{p-1}u_k)_{k=1}^N$, $\mu \in \{\pm1\}$, and $\Delta_\mathcal{G}^\gamma$ is the Laplace  operator on $\mathcal{G}$, which is precisely defined below. 
For \eqref{NLS}, the nonlinearity is said to be defocusing if $\mu =1$ and focusing if $\mu=-1$. 
In physics, the metric graphs, including star graphs, 
were introduced as a simplified model of 
quasi-one-dimensional region, referred as \textit{quantum graphs} in the quantum theory. 
There are extensive studies from experimental and theoretical aspects \cite{BerKuc13, BlaExnHav08}. 
The nonlinear Schr\"odinger equation on metric graphs is the corresponding nonlinear dynamical model, aimed to understand the corresponding nonlinear behavior on graphs. For more details of theoretical background in physics, see \cite{GnuSmiDer11, 
SMSSN10} for instance. 
\par
To state our main results for \eqref{NLS}, 
we give definitions of function spaces on $\mathcal{G}$ and the Laplacian $\Delta_{\mathcal{G}}^\gamma$. We define Lebesgue spaces and Sobolev spaces on the star graph $\mathcal{G}$ as follows:
\begin{align*}
	L^q(\mathcal{G})
		:= \bigoplus_{k=1}^N L^q(0,\infty),\quad
	W^{s,q}(\mathcal{G})
		:= \bigoplus_{k=1}^N W^{s,q}(0,\infty),
\end{align*}
whose norms are defined by
\begin{align*}
	&\|\bm{\bof}\|_{L^q(\mathcal{G})}^{q}
		:= \sum_{k=1}^N \|f_k\|_{L^q(0,\infty)}^{q},\quad 
	\|\bm{\bof}\|_{L^\infty(\mathcal{G})}
		:= \sum_{k=1}^N \|f_k\|_{L^\infty(0,\infty)},
	\\
	&\|\bm{\bof}\|_{W^{s,q}(\mathcal{G})}^{q}
		:= \sum_{k=1}^N \|f_k\|_{W^{s,q}(0,\infty)}^{q},
\end{align*}
for $1 \leq q < \infty$ and $s=1,2$. 
We use the usual notation $H^s(\mathcal{G}) := W^{s,2}(\mathcal{G})$, which has the inner product
\begin{align*}
	(\bm{f},\bm{g})_{H^{s}(\mathcal{G})} := \sum_{k=1}^{N} (f_{k},g_{k})_{H^{s}(0,\infty)}.
\end{align*}
We use $H_c^s(\mathcal{G})$ ($s=1,2$) to denote the set of the functions in $H^s(\mathcal{G})$ which are continuous at the origin. 
The Laplacian $\Delta_\mathcal{G}^\gamma$ is defined by 
\begin{align*}
	\mathcal{D}(\Delta_\mathcal{G}^\gamma) &:= 
	\left\{ \bm{f}=(f_k)_{k=1}^{N} \in H_{c}^{2}(\mathcal{G}) : 
	\sum_{k=1}^{N} f'_{k}(0+)=N\gamma f_{1}(0) \right\},
	\\
	\Delta_\mathcal{G}^\gamma \bm{f}&:=(\partial_{xx} f_k)_{1\leq k \leq N}.
\end{align*}
When $\gamma=0$, the condition at the origin is called the Kirchhoff condition. 
When $\gamma>0$ (resp. $\gam<0$), it is called repulsive (resp. attractive) delta condition.
%
It is known that $\Dgg$ is a self-adjoint operator on $L^2(\boG)$ (see \cite{KosSch06}) and thus the Schr\"{o}dinger propagator $e^{it\Dgg}$ is well defined by the Stone theorem. 
When $\gam\ge 0$, the spectrum $\Dgg$ only consists of continuous spectrum, while 
when $\gam<0$ there is an eigenvalue $-\gam^2$ with the explicit eigenfunction 
$(e^{\gam x},\cdots ,e^{\gam x})$. 
\par
%
%
It is known that \eqref{NLS} is locally well-posed in $H_c^1(\mathcal{G})$, 
and the mass $M$ and energy $E_{\gamma}$ defined by
\begin{align}
\tag{Mass}
	&M(\bm{u}(t)):= \|\bm{u}(t)\|_{L^2(\mathcal{G})}^2, 
\\
\tag{Energy}
	&E_\gamma(\bm{u}(t))
		:= \frac{1}{2}\|\partial_x\bm{u}(t)\|_{L^2(\mathcal{G})}^2 + \frac{N\gamma}{2}|u_1(0)|^2 + \frac{\mu}{p+1}\|\bm{u}(t)\|_{L^{p+1}(\mathcal{G})}^{p+1}
\end{align}
are conserved by the flow. 
More precisely, we have the following result.
\begin{theorem}[Local well-posedness of \eqref{NLS} \cite{AdaCacFinNoj14, AngGol18, CacFinNoj17}]
Let $p > 1$ and $\gamma \geq 0$.
For any $\bm{u_0} \in H_c^1(\mathcal{G})$, there exist $T_\text{max} \in (0,\infty]$ and $T_\text{min} \in [-\infty,0)$ such that \eqref{NLS} has a unique solution $\bm{u} \in C((T_\text{min},T_\text{max});H_c^1(\mathcal{G}))$, 
where $(T_\text{min},T_\text{max})$ is called maximal lifespan of $\bm{u}$.
For each compact interval $I \subset (T_\text{min},T_\text{max})$, the mapping $H_c^1(\mathcal{G}) \ni \bm{u_0} \mapsto \bm{u} \in C(I;H_c^1(\mathcal{G}))$ is continuous.
Moreover, the solution $\bm{u}$ has the following blow-up alternative: If $T_\text{max} < \infty$ (\text{resp.} $T_\text{min} > -\infty$), then
\begin{align*}
	\lim_{t \nearrow T_\text{max}}\|\bm{u}(t)\|_{H^1(\mathcal{G})}
		= \infty\ \ \ 
	\left(\text{resp.} \lim_{t \searrow T_\text{min}}\|\bm{u}(t)\|_{H^1(\mathcal{G})}
		= \infty \right).
\end{align*}
Furthermore, the solution $\bm{u}$ to \eqref{NLS} preserves the mass $M$ and energy $E_{\gamma}$ with respect to time $t$. 
\end{theorem}

%

%
%
%
%
In the present paper, we are interested in the global behavior of the solutions to \eqref{NLS} in the $L^{2}$-supercritical case, i.e., $p>5$. 
For the corresponding equation on the real line, this has been extensively studied 
for the free case (without potentials) by \cite{Nak99, AkaNaw13, FanXieCaz11, Gue, Guo}, and for the case with the potential by \cite{BanVis16, IkeInu17}. 
Let us overview these studies. 
In the defocusing case ($\mu=1$), 
all the $H^1$-solutions are global in time and scatters both direction in time. 
On the other hand, in the focusing case ($\mu=-1$), various behaviors are possible to occur. 
One typical example is the \textit{standing wave} solution $e^{i\om t} \vph$ with $\om>0$ and $\vph\in H^1(\R)$, which is an example of non-scattering solution. 
Then the function $\vph$ is a critical point of the action functional 
\begin{align*}
	S_{\omega,\gamma}^{\mathrm{line}}(f)&:= 
\frac 12 \nor{\rd_x f}{L^2(\R)}^2 +\gamma|f(0)|^2
-\frac{1}{p+1} \nor{f}{L^{p+1}(\R)}^{p+1} + \frac \om 2 \nor{f}{L^2(\R)}^2.
\end{align*}
The potential well theory has revealed that the following minimization problem plays an important role in describing the global-in-time behavior:
\begin{equation}\label{min:line}
	\mathfrak{n}_{\omega,\gamma}^{\mathrm{line}}
		 := \inf\{S_{\omega,\gamma}^\mathrm{line}(f) : f \in H^1(\mathbb{R}) \setminus \{0\},\,K_{\gamma}^{\mathrm{line}}(f) = 0\},
\end{equation}
where 
\begin{align*}
	K_{\gamma}^{\mathrm{line}}(f)&:=2\|\partial_xf\|_{L^2(\mathbb{R})}^2 + 2\gamma|f(0)|^2 - \frac{p-1}{p+1}\|f\|_{L^{p+1}(\mathbb{R})}^{p+1},
\end{align*}
called virial functional. 
In detail, the energy space with action below $\frak{n}_{\om,\gam}^\li$ 
is separated into two invariant sets under the flow. One of which exhibits stable dynamics, scatters on both directions in time, while the other behaves unstably, where the $H^1(\R)$-norms diverges in finite or at infinite time. 
The aim of the present paper is to show the corresponding result for \eqref{NLS}.\par
Let us observe the properties on minimization problem in our case: 
\begin{align*}
	\mathfrak{n}_{\omega,\gamma}
		& := \inf\{S_{\omega,\gamma}(\bm{f}) : \bm{f} \in H_c^1(\mathcal{G})\setminus\{\bm{0}\},\ K_{\gamma}(\bm{f}) = 0\},
\end{align*}
where the action and virial functionals are defined by
\begin{align*}
	S_{\omega,\gamma}(\bm{f})&:= E_{\gamma}(\bm{f}) + \frac{\omega}{2} M(\bm{f}),
	\\
	K_{\gamma}(\bm{f})&:=2\|\partial_x\bm{f}\|_{L^2(\mathcal{G})}^2 + N\gamma|f_1(0)|^2 - \frac{p-1}{p+1}\|\bm{f}\|_{L^{p+1}(\mathcal{G})}^{p+1}.
\end{align*}
The similar minimization problem is investigated in \cite{AdaCacFinNoj14}, 
and the following holds.

\begin{proposition}[\cite{AdaCacFinNoj14}]
\label{Th:Minimization problem}
Let $p > 1$, $\gamma \geq 0$, and $\omega > 0$.
Then $\mathfrak{n}_{\omega,\gamma} =\mathfrak{n}_{\omega,0}^\mathrm{line}$ holds and $\mathfrak{n}_{\omega,\gamma}$ is not attained. 
\end{proposition}

In \cite{AdaCacFinNoj14}, the constraint is not in terms of the virial functional but the Nehari functional. However the assertion follows essentially in the similar way to \cite{AdaCacFinNoj14} by using the property on best constant of the Gagliardo--Nirenberg inequality shown in Section \ref{secGN} below. 
As shown in \cite{AdaCacFinNoj14}, the infimum is not attained. 
%
The mechanism of the breakdown of attainability is the same as that of the Schr\"odinger equation with a repulsive delta potential on the real line (see \cite{FukJea08}), but 
it should be noted that on the graph case, this occurs even in the Kirchhoff case $\gam=0$.

\subsection{Main Results}
We first introduce several notions and notations.
\begin{definition}[Scattering, blow-up, and grow-up]
Let $\bm{u} \in C( I;H_c^1(\mathcal{G}))$ be a solution to \eqref{NLS}, where $I=(T_{\min},T_{\max})$.
\begin{itemize}
\item (Scattering)
We say that $\bm{u}$ scatters in positive time (resp. negative time) when $T_\text{max} = \infty$ (resp. $T_\text{min} =-\infty$) and there exists $\bm{\psi_+} \in H_c^1(\mathcal{G})$ (resp. $\bm{\psi_-} \in H_c^1(\mathcal{G})$) such that
\begin{align*}
	\lim_{t \rightarrow +\infty}\|\bm{u}(t) - e^{it\Delta_\mathcal{G}^\gamma}\bm{\psi_+}\|_{H^1(\mathcal{G})}
		= 0\ \ \ 
	\left( \text{resp. }\lim_{t \rightarrow -\infty}\|\bm{u}(t) - e^{it\Delta_\mathcal{G}^\gamma}\bm{\psi_-}\|_{H^1(\mathcal{G})}
		= 0 \right).
\end{align*}
\item (Blow-up)
We say that $\bm{u}$ blows up in positive time (resp. negative time) when $T_\text{max} < \infty$ (resp. $T_\text{min} > -\infty$).
\item (Grow-up)
We say that $\bm{u}$ grows up in positive time (resp. negative time) when $T_\text{max} = \infty$ (resp. $T_\text{min} = -\infty$) and
\begin{align*}
	\limsup_{t \rightarrow +\infty}\|\bm{u}(t)\|_{H^1(\mathcal{G})}
		= \infty\ \ \ 
	\left( \text{resp. }\limsup_{t \rightarrow -\infty}\|\bm{u}(t)\|_{H^1(\mathcal{G})}
		= \infty \right).
\end{align*}
\end{itemize}
\end{definition}
For $\om>0$ we define $Q_{\omega}(x)\in H^1(\R)$ by the unique positive solution to
$$
-\rd_{xx} Q_{\omega} + \om Q_{\omega} = Q_{\omega}^{p}, \quad x \in \mathbb{R},
$$
which is explicitly written as 
\begin{align*}
	Q_{\omega}(x) :=
	\left[\frac{(p+1)\omega}{2}\sech^2\left\{\frac{(p-1)\sqrt{\omega}}{2}|x|\right\}\right]^\frac{1}{p-1}.
\end{align*}
Note that this is a minimizer of \eqref{min:line} with $\gam=0$. We set $Q:=Q_{1}$ for simplicity. 

Now we state our main results. 
We first state the result of the focusing case. 
\begin{theorem}[Focusing with frequency $\omega$]
\label{Th: Non-radial case}
Let $p > 5$, $\gamma \geq 0$, $\mu = -1$, and $\omega > 0$.
Let $\bm{u}$ be a solution to \eqref{NLS} with initial data $\bm{u_0} \in H_c^1(\mathcal{G})$. 
Assume that $S_{\omega,\gamma}(\bm{u_0}) < \frak{n}_{\om,0}^\li$ 
for some $\omega > 0$. Then the following are true.
\begin{itemize}
\item (Scattering)
If $\bm{u_0}$ satisfies $K_\gamma(\bm{u_0}) \geq 0$, then the solution $\bm{u}$ to \eqref{NLS} scatters in both time directions.
\item (Blow-up or grow-up)
If $\bm{u_0}$ satisfies $K_\gamma(\bm{u_0}) < 0$, then the solution $\bm{u}$ to \eqref{NLS} blows up or grows up in both time directions. Additionally, if $x\bm{u_0} \in L^2(\mathcal{G})$, the solution $\bm{u}$ blows up in both time directions. 
\end{itemize}
\end{theorem}
%
%
We can also obtain the $\omega$-independent form of the statement. 
We define the first order Sobolev norm related to $\Delta_{\mathcal{G}}^{\gamma}$ by
\begin{align*}
	\|\bm{f}\|_{\dot{H}_{\gamma}^{1}(\mathcal{G})}^{2} 
	:=\|(-\Delta_{\mathcal{G}}^{\gamma})^{\frac{1}{2}} \bm{f}\|_{L^{2}(\mathcal{G})}^{2}
	=\|\partial_{x}\bm{f}\|_{L^{2}(\mathcal{G})}^{2} +N\gamma |f_{1}(0)|^{2}
\end{align*}
for $\bm{f} \in H_{c}^{1}(\mathcal{G})$. 
Also we denote the mass and energy for the line case by
\begin{align*}
	M^\mathrm{line}(f)
		&:= \|f\|_{L^2(\mathbb{R})}^2, 
	\\
	E_\gamma^\mathrm{line}(f)
		&:= \frac{1}{2}\|\partial_x f\|_{L^2(\mathbb{R})}^2 + \gamma |f(0)|^{2}- \frac{1}{p+1}\|f\|_{L^{p+1}(\mathbb{R})}^{p+1},
\end{align*}
for $f:\R\to\C$ and $\gam\in\R$. Then we have the following.
\begin{corollary}[Focusing without frequency $\omega$]
\label{Cor: Non-radial case}
Let $p > 5$, $\gamma \geq 0$, $\mu = -1$, and $s_c= \frac 12 -\frac{2}{p-1}$. 
Let $\bm{u}$ be a solution to \eqref{NLS} with initial data $\bm{u_0} \in H_c^1(\mathcal{G})$. We assume $M(\bm{u_{0}})^{(1-s_c)/s_c}E_\gamma(\bm{u_{0}}) < M^\mathrm{line}(Q)^{(1-s_c)/s_c}E_0^\mathrm{line}(Q)$. 
Then we have the following results:
\begin{itemize}
\item (Scattering)
If $\bm{u_0}$ satisfies $$\| \bm{u_{0}}\|_{L^{2}(\mathcal{G})}^{1-s_c}\| \bm{u_{0}}\|_{\dot{H}_{\gamma}^{1}(\mathcal{G})}^{s_c} < \|Q\|_{L^{2}(\mathbb{R})}^{1-s_c}\|Q\|_{\dot{H}^{1}(\mathbb{R})}^{s_c},$$ then the solution $\bm{u}$ to \eqref{NLS} scatters in both time directions.
\item (Blow-up or grow-up)
If $\bm{u_0}$ satisfies $$\| \bm{u_{0}}\|_{L^{2}(\mathcal{G})}^{1-s_c}\| \bm{u_{0}}\|_{\dot{H}_{\gamma}^{1}(\mathcal{G})}^{s_c} > \|Q\|_{L^{2}(\mathbb{R})}^{1-s_c}\|Q\|_{\dot{H}^{1}(\mathbb{R})}^{s_c},$$ then the solution $\bm{u}$ to \eqref{NLS} blows up or grows up in both time directions. Additionally, if $x\bm{u_0} \in L^2(\mathcal{G})$, the solution $\bm{u}$ blows up in both time directions. 
\end{itemize}
\end{corollary}
In the defocusing case, we have the scattering result.
\begin{theorem}[Defocusing case]
\label{thm1.11}
Let $p > 5$, $\gamma \geq 0$, and $\mu = 1$.
Let $\bm{u}$ be a solution to \eqref{NLS} with initial data $\bm{u_0} \in H_c^1(\mathcal{G})$.
Then, the solution $\bm{u}$ to \eqref{NLS} scatters in both time directions.
\end{theorem}

We mention the case under restriction of radial symmetry. We define the energy space with radial symmetry by
$$
H^1_\text{rad} (\boG) := \{ \bof = (f_k)_{k=1}^N\in H^1_c(\boG) : f_1=\cdots =f_N\}.
$$
We define
\begin{align*}
	\mathfrak{r}_{\omega,\gamma}
		& := \inf\{S_{\omega,\gamma}(\bm{f}) : \bm{f} \in H_\mathrm{rad}^1(\mathcal{G})\setminus\{\bm{0}\},\ K_{\gamma}(\bm{f}) = 0\}.
\end{align*}
We note that $\mathfrak{r}_{\om,\gam} > \mathfrak{n}_{\om,\gam}$ and there exists an excited state which attains $\mathfrak{r}_{\om,\gam}$ (see \cite{AdaCacFinNoj14}). Then we have the following result.

\begin{theorem}[Focusing and Radial case]\label{Th: Radial case}
Let $p > 5$, $\gamma \geq 0$, $\mu = -1$, and $\omega > 0$.
Let $\bm{u}$ be a solution to \eqref{NLS} with initial data $\bm{u_0} \in H_\text{rad}^1(\mathcal{G})$, where we note that $\bm{u}(t)  \in H_\text{rad}^1(\mathcal{G})$.
We assume that $S_{\omega,\gamma}(\bm{u_0}) <\mathfrak{r}_{\omega,\gamma}$ for some $\omega > 0$.
\begin{itemize}
\item (Scattering)
If $\bm{u_0}$ satisfies $K_\gamma(\bm{u_0}) \geq 0$, then the solution $\bm{u}$ to \eqref{NLS} scatters in both time directions.
\item (Blow-up or grow-up)
If $\bm{u_0}$ satisfies $K_\gamma(\bm{u_0}) < 0$, then the solution $\bm{u}$ to \eqref{NLS} blows up or grows up in both time directions. Additionally, if $x\bm{u_0} \in L^2(\mathcal{G})$, the solution $\bm{u}$ blows up in both time directions. 
\end{itemize}
\end{theorem}

%




At the end of this subsection, let us mention some previous related results. 
The existence of standing waves for \eqref{NLS} on star graphs is studied in \cite{AdaCacFinNoj14}, where not only the ground states, but also the excited states are constructed explicitly (see also \cite{
ACFN14}). Also the orbital stability/instability of these solutions is investigated in detail via variational method or so-called the extension theory \cite{AdaCacFinNoj14, ACFN14, AdaCacFinNoj16, 
KaiPel18SS, KaiPel18HS, AngGol18Ins, AngGol18, Kai19, GolOht20}. 
The corresponding study has been extended to various types of metric graphs, such as tadpole graphs \cite{CacFinNoj15, NojPelSha15, NojPel20}, periodic graphs \cite{PelSch17, Pan18, Dov19}, compact metric graphs \cite{Dov18, CacDovSer18, DovGhiMicPis20}, and other general graphs \cite{AdaSerTil15, SerTen16, AdaSerTil16, DovSerTil20, BerMerPel21}, where it is revealed that the topology of graphs makes influence on the existence and stability, and a rich variety of structures can be observed. The dispersive estimate for the Schr\"odinger propagator on metric graphs has been established in \cite{BanIgn11, BanIgn14, GreIgn19}, which immediately implies the scattering of small solutions of \eqref{NLS} with short-range nonlinearities. In the case of a special long-range nonlinearity $p=3$, \cite{AIMMU20} constructs a solution with modified scattering asymptotics, which is an example of non-scattering solutions. It is also shown by \cite{AokInuMiz21} that no solution scatters to a standing wave for the long-range nonlinearities. 

\subsection{Idea of proof and our contribution}

The proof of the scattering part in Theorem \ref{Th: Non-radial case} is  based on a concentration compactness and rigidity argument by Kenig and Merle \cite{KenMer06}. 
One of the crucial step is to establish so-called linear profile decomposition. 
The basic philosophy is to detect all the non-compact factors of $e^{it\Dgg}$ and 
to decompose all the bounded sequences in $H^1_c(\boG)$ into profiles which moves different speeds of space-time \textit{shift}. 
For example, 
in the whole line case, 
any bounded sequence $\{\varphi_{n}\}$ in $H^{1}(\mathbb{R})$ is decomposed, by taking  subsequence if necessary, into 
\begin{align}
\label{eq0627-1}
	\varphi_{n} = \sum_{j=1}^{J} e^{it_{n}^{j}\rd_{xx}} \tau_{x_{n}^{j}} \psi^{j} + w_{n}^{J},
\end{align}
where $t_{n}^{j}, x_{n}^{j} \in \mathbb{R}$,  $\psi^{j}, w_{n}^{J} \in H^{1}(\mathbb{R})$, $\tau_{x_{n}^{j}}\psi^{j}=\psi^{j}(\cdot-x_{n}^{j})$,
and $w_n^J$ converges weakly to $0$ as $n\to\infty$. 
In the case of star graph, however, 
it is not clear how to represent spatial translation, and accordingly to 
describe spatial non-compactness in the profile decomposition. 
The novelty of the present paper is 
to give a quite systematic representation of profiles, 
based on a new idea of decomposition of functions on graphs. 
In more detail, we introduce the operator $\mathcal{Q}$ which maps 
a function in $H^1_c(\boG)$ to a family of $N-1$ odd functions and one even function on line. 
A remarkable feature is that this decomposition \textit{commutes} with the Schr\"odinger propagators; 
$\mathcal{Q}e^{it\Dgg} \bof = e^{it\Delta_{\mathbb{R}}^{\gamma}} \mathcal{Q}\bof$ for $\bof\in H^1_c(\boG)$ and $t\in\R$
, so that our linear profile decomposition follows largely from the corresponding results on line. \par
We remark that 
\cite{ACFN14} also establishes the concentration-compactness lemma for the purpose of optimization problem for the Gagliardo-Nirenberg inequality, where 
the runaway profiles are 
constructed from single function with cut-off around the origin. 
Compared to that, 
our approach is better suited to the Laplace operator and accordingly the Schr\"odinger propagator. 
For this reason, the idea in the present article may well be applied to other studies on star graphs. 

The blow-up part in Theorem \ref{Th: Non-radial case} can be shown by the similar method to \cite{AkaNaw13} and thus we omit the proof. Moreover, the scattering result in the defocusing case, Theorem \ref{thm1.11}, can be proven in the same way as the scattering part in Theorem \ref{Th: Non-radial case}. Thus we also omit the proof. Theorem \ref{Th: Radial case} follows directly from \cite{IkeInu17} since radial functions on the star graph is reduced to even functions on $\mathbb{R}$.

%
%


\subsection{Remarks and organization of the paper}

We sometimes omit the domains of norms, for example $\|\bm{f}\|_{L^{q}} := \|\bm{f}\|_{L^{q}(\mathcal{G})}$ and $\|f\|_{L^{q}} := \|f\|_{L^{q}(\mathbb{R})}$, if it is obvious. 

In what follows, we only consider the case $N=3$ for simplicity, while the similar argument also works in the general case. As stated before, we only give the proof in the focusing case ($\mu=-1$), and thus we set $\mathcal{N}:=\mathcal{N}_{-1}$. 

This paper is organized as follows. 
In Section 2, we discuss the variational aspect of \eqref{NLS}, including the optimization of the Gagliardo-Nirenberg inequality and the potential well theory. 
In Section 3, we prepare some technical lemmas, such as properties on linear solutions and localized virial identities. 
In Section 4, we introduce linear profile decomposition for Schr\"odinger propagator and give the proof. 
In Section 5, the corresponding nonlinear profile decompositions are discussed. 
The proof of scattering part of the main theorem is given in Section 6, where we apply concentration compactness and rigidity argument. 
We state the result in the case of radial symmetry. Actually, there are many symmetries for functions on a star graph. We discuss them in Appendix \ref{appA}.

\section{Variational argument}\label{Sec:Var}

\subsection{The Gagliardo--Nirenberg inequality}
\label{secGN}

We consider the Gagliardo--Nirenberg type inequalities:
\begin{align}
\label{G-N inequality}
	&\|\bm{f}\|_{L^{p+1}(\mathcal{G})}^{p+1}
		\leq C_\mathrm{GN}(\gamma)\|\bm{f}\|_{L^2(\mathcal{G})}^\frac{p+3}{2}
		\|\bm{f}\|_{\dot{H}_\gamma^1(\mathcal{G})}^\frac{p-1}{2},
		 \quad (\bm{f} \in H_c^1(\mathcal{G})), \\ \notag
	&\|g\|_{L^{p+1}(\mathbb{R})}^{p+1}
		\leq C_\mathrm{GN}^\mathrm{line}\|g\|_{L^2(\mathbb{R})}^\frac{p+3}{2}\|\partial_xg\|_{L^2(\mathbb{R})}^\frac{p-1}{2},  \quad (g \in H^1(\mathbb{R})),
\end{align}
where $C_\mathrm{GN}(\gamma)$ and $C_\mathrm{GN}^\mathrm{line}$ are the best constants respectively.

\begin{lemma}
\label{lem2.1}
Let $p > 1$ and $\gamma > 0$. Then we have
$C_\mathrm{GN}(\gamma) = C_\mathrm{GN}(0)$.
\end{lemma}

\begin{proof}
Set $\bm{\varphi}_{\lambda}(x):=\bm{\varphi}(\lambda x)$ for $\lambda>0$. For any $\bm{\varphi} \in H_c^1(\mathcal{G})\setminus\{\bm{0}\}$, we have
\begin{align*}
	\frac{1}{C_\mathrm{GN}(\gamma)} 
	&\leq \frac{\|\bm{\varphi}_{\lambda}\|_{L^2(\mathcal{G})}^\frac{p+3}{2} (\| \partial_{x}\bm{\varphi}_{\lambda}\|_{L^2(\mathcal{G})}^{2} + 3\gamma|\varphi_{1}(0)|^{2} )^\frac{p-1}{4}}{\|\bm{\varphi}_{\lambda}\|_{L^{p+1}(\mathcal{G})}^{p+1}}
	\\
	&=\frac{\|\bm{\varphi}\|_{L^2(\mathcal{G})}^\frac{p+3}{2} (\| \partial_{x}\bm{\varphi}\|_{L^2(\mathcal{G})}^{2} + 3\gamma \lambda^{-1}|\varphi_{1}(0)|^{2} )^\frac{p-1}{4}}{\|\bm{\varphi}\|_{L^{p+1}(\mathcal{G})}^{p+1}}
	\\
	& \to \frac{\|\bm{\varphi}\|_{L^2(\mathcal{G})}^\frac{p+3}{2} \| \partial_{x}\bm{\varphi}\|_{L^2(\mathcal{G})}^\frac{p-1}{2}}{\|\bm{\varphi}\|_{L^{p+1}(\mathcal{G})}^{p+1}}
\end{align*}
as $\lambda \to \infty$. Taking the infimum for $\bm{\varphi}$, we get
\begin{align*}
	\frac{1}{C_\mathrm{GN}(\gamma)}  \leq \frac{1}{C_\mathrm{GN}(0)}. 
\end{align*}
The inverse inequality is obvious since $\gamma >0$. 
\end{proof}

\begin{lemma}
\label{lem2.2}
Let $p > 1$ and $\gamma \geq 0$.
Then, $C_\mathrm{GN}(\gamma) = C_\mathrm{GN}^\mathrm{line}$.
\end{lemma}

\begin{proof}
First, we prove $C_\mathrm{GN}(\gamma) \geq C_\mathrm{GN}^\text{line}$.
We consider a sequence $\{\bm{\varphi_{n}}\}\subset H_c^1(\mathcal{G})$ such that 
\begin{align*}
	(\bm{\varphi_{n}})_1(x) =\psi(x) Q(x-n), 
	\quad
	(\bm{\varphi_{n}})_k(x) = 0
\end{align*}
for $k\neq 1$, where $\psi \in C^\infty((0,\infty))$ satisfies $\psi(x)=0$ for $x\in (0,1)$ and $\psi(x)=1$ for $x\in (2,\infty)$. 
We have
\begin{align*}
	\frac{1}{C_\mathrm{GN}(\gamma)}
		\leq \frac{\|\bm{\varphi_{n}}\|_{L^2(\mathcal{G})}^\frac{p+3}{2}\|\bm{\varphi_{n}}\|_{\dot{H}_\gamma^1(\mathcal{G})}^\frac{p-1}{2}}{\|\bm{\varphi_{n}}\|_{L^{p+1}(\mathcal{G})}^{p+1}} 
		\to \frac{\|Q\|_{L^2(\mathbb{R})}^\frac{p+3}{2}\|\partial_xQ\|_{L^2(\mathbb{R})}^\frac{p-1}{2}}{\|Q\|_{L^{p+1}(\mathbb{R})}^{p+1}}
		= \frac{1}{C_\mathrm{GN}^\text{line}}
\end{align*}
as $n \to \infty$. 
Thus, $C_\mathrm{GN}(\gamma) \geq C_\mathrm{GN}^\text{line}$ holds.

Next, we show $C_\mathrm{GN}(\gamma) \leq C_\mathrm{GN}^\text{line}$.
By Lemma \ref{lem2.1}, it suffices to prove that $C_\mathrm{GN}(0) \leq C_\mathrm{GN}^\text{line}$.
If $\bm{f}^\ast$ denotes the symmetric rearrangement of $\bm{f}\in H_c^1(\mathcal{G})$, then $\bm{f}^\ast \in H_{\text{rad}}^1(\mathcal{G})$ and 
\begin{align*}
	\|\bm{f}\|_{L^2(\mathcal{G})}
		= \|\bm{f}^\ast\|_{L^2(\mathcal{G})}, \quad
	\|\partial_x\bm{f}\|_{L^2(\mathcal{G})}
		\geq \tfrac{2}{3}\|\partial_x\bm{f}^\ast\|_{L^2(\mathcal{G})}, \quad
	\|\bm{f}\|_{L^{p+1}(\mathcal{G})}
		= \|\bm{f}^\ast\|_{L^{p+1}(\mathcal{G})}
\end{align*}
hold (see \cite{AdaCacFinNoj14}).
Therefore, it follows that
\begin{align*}
	\frac{1}{C_\mathrm{GN}(0)}
		& = \inf_{\bm{f} \in H_c^1(\mathcal{G})}\frac{\|\bm{f}\|_{L^2(\mathcal{G})}^\frac{p+3}{2}\|\partial_x\bm{f}\|_{L^2(\mathcal{G})}^\frac{p-1}{2}}{\|\bm{f}\|_{L^{p+1}(\mathcal{G})}^{p+1}}
		\geq \left(\frac{2}{3}\right)^\frac{p-1}{2}\inf_{\bm{f} \in H_\text{rad}^1(\mathcal{G})}\frac{\|\bm{f}\|_{L^2(\mathcal{G})}^\frac{p+3}{2}\|\partial_x\bm{f}\|_{L^2(\mathcal{G})}^\frac{p-1}{2}}{\|\bm{f}\|_{L^{p+1}(\mathcal{G})}^{p+1}} \\
		& = \left(\frac{2}{3}\right)^\frac{p-1}{2}\inf_{f \in H_\text{rad}^1(\mathbb{R})}\frac{(3/2)^\frac{p+3}{4}\|f\|_{L^2(\mathbb{R})}^\frac{p+3}{2} \cdot (3/2)^\frac{p-1}{4}\|\partial_xf\|_{L^2(\mathbb{R})}^\frac{p-1}{2}}{(3/2)\|f\|_{L^{p+1}(\mathbb{R})}^{p+1}}
		= \frac{1}{C_\mathrm{GN}^\text{line}},
\end{align*}
which implies $C_\mathrm{GN}^\text{line} \geq C_\mathrm{GN}(0)$.
\end{proof}


\subsection{Potential well sets}

We recall the Pohozaev identity, which can be found in e.g. \cite[Corollary 8.1.3]{Caz03}.

\begin{lemma}[The Pohozaev identity]
\label{lem2.3}
We have 
\begin{align*}
	\frac{1}{p+3}\| Q\|_{L^2(\mathbb{R})}^2 = \frac{1}{p-1}\|\partial_x Q\|_{L^2(\mathbb{R})}^2 = \frac{1}{2(p+1)}\| Q\|_{L^{p+1}(\mathbb{R})}^{p+1} 
\end{align*}
and, in particular, if $p>5$, then
\begin{align*}
	M^{\mathrm{line}}(Q)= \frac{2(p+3)}{p-5} E_{0}^{\mathrm{line}}(Q).
\end{align*}
\end{lemma}

Recall that $s_{c}=\frac{1}{2}-\frac{2}{p-1}=\frac{p-5}{2(p-1)}$, and thus $1-s_{c}=\frac{p+3}{2(p-1)}$. 

\begin{proposition}
\label{prop2.4.0}
Let $p > 5$, $\gamma \geq 0$, and $\mu = -1$.
Let $\bm{f} \in H_c^1(\mathcal{G})$. Then the following are equivalent:
\begin{enumerate}
\renewcommand{\theenumi}{(\roman{enumi})}
\item $S_{\omega,\gamma}(\bm{f}) \leq \mathfrak{n}_{\omega,\gamma}(=\mathfrak{n}_{\omega,0}^{\mathrm{line}})$ for some $\omega>0$.
\item $M(\bm{f})^{(1-s_c)/s_c}E_{\gamma}(\bm{f}) \leq M^{\mathrm{line}}(Q)^{(1-s_c)/s_c} E_{0}^{\mathrm{line}}(Q)$.
\end{enumerate}
\end{proposition}

\begin{proof}
Let $F(M):=M^{\mathrm{line}}(Q)^{\frac{1-s_{c}}{s_{c}}} E_{0}^{\mathrm{line}}(Q)M^{-\frac{1-s_{c}}{s_{c}}}$ for $M>0$. 
Set $M_{\omega}:=M^{\mathrm{line}}(Q_{\omega,0})=\omega^{-\frac{p-5}{2(p-1)}}M^{\mathrm{line}}(Q)$ for $\omega>0$.
Then we have
\begin{align*}
	F(M_{\omega})=\omega^{\frac{p-5}{2(p-1)}\frac{1-s_{c}}{s_{c}}} E_{0}^{\mathrm{line}}(Q)
	=\omega^{\frac{p+3}{2(p-1)}} E_{0}^{\mathrm{line}}(Q)
\end{align*}
and
\begin{align*}
	F'(M_{\omega}) 
	= -\frac{1-s_{c}}{s_{c}}M^{\mathrm{line}}(Q)^{\frac{1-s_{c}}{s_{c}}} E_{0}^{\mathrm{line}}(Q)M_{\omega}^{-\frac{1}{s_{c}}} 
	=- \frac{\omega}{2} 
\end{align*}
by Lemma \ref{lem2.3}. Set $G_{\omega}(M):=-\frac{\omega}{2} M + \mathfrak{n}_{\omega,0}^{\mathrm{line}}$. Since $\mathfrak{n}_{\omega,0}^{\mathrm{line}} =\omega^{\frac{p+3}{2(p-1)}}E_{0}^{\mathrm{line}}(Q)+\frac{\omega}{2} \omega^{-\frac{p-5}{2(p-1)}}M^{\mathrm{line}}(Q)$, we have (a) $G_{\omega}(M_{\omega})= F(M_{\omega})$.
Thus $G_{\omega}(M)$ is the tangent line of the graph of $F$ at $M=M_{\omega}$. Since $F$ is strictly convex on $(0,\infty)$, we obtain (b) $G_{\omega}(M) \leq F(M)$ for all $M>0$ and  $\omega>0$.

We prove (ii)  from (i). If $\bm{f} \in H_c^1(\mathcal{G})$ satisfies (i), that is, $E_{\gamma}(\bm{f}) \leq -\frac{\omega}{2} M(\bm{f}) + \mathfrak{n}_{\omega,0}^{\mathrm{line}}=G_{\omega}(M(\bm{f}))$ for some $\omega>0$, then we get $E_{\gamma}(\bm{f}) \leq F(M(\bm{f}))$ by (b). This implies (i). 
Next, we show (i)  from (ii). Assume that $\bm{f} \in H_c^1(\mathcal{G})$ satisfies (ii), that is, $E_{\gamma}(\bm{f}) \leq M^{\mathrm{line}}(Q)^{\frac{1-s_{c}}{s_{c}}} E_{0}^{\mathrm{line}}(Q)M(\bm{f})^{-\frac{1-s_{c}}{s_{c}}}=F(M(\bm{f}))$. Since $M_{\omega}$ can take arbitrary positive number by moving $\omega$, there exists $\omega>0$ such that $M_{\omega}=M(\bm{f})$. Thus, we obtain $E_{\gamma}(\bm{f}) \leq F(M(\bm{f})) = G_{\omega}(M(\bm{f}))$ by (a), which means (ii).
\end{proof}

As a corollary, we have the following:
\begin{corollary}
\label{cor2.5}
Let $p > 5$, $\gamma \geq 0$, and $\mu = -1$.
Let $\bm{f} \in H_c^1(\mathcal{G})$. Then the following are equivalent:
\begin{enumerate}
\renewcommand{\theenumi}{(\roman{enumi})}
\item $S_{\omega,\gamma}(\bm{f}) < \mathfrak{n}_{\omega,\gamma}(=\mathfrak{n}_{\omega,0}^{\mathrm{line}})$ for some $\omega>0$.
\item $M(\bm{f})^{(1-s_c)/s_c}E_{\gamma}(\bm{f}) < M^{\mathrm{line}}(Q)^{(1-s_c)/s_c} E_{0}^{\mathrm{line}}(Q)$.
\end{enumerate}
\end{corollary}

\begin{proof}
Let $F$ and $G_{\omega}$ be defined in the proof of Proposition \ref{prop2.4.0}. 
The statement follows form the strict convexity of $F$ and the fact that $G_{\omega}$ is the tangent line. 
\end{proof}

\begin{proposition}
\label{prop2.6.0}
Let $p > 5$, $\gamma \geq 0$, and $\mu = -1$.
Assume that $\bm{f} \in H_c^1(\mathcal{G})$ satisfies $S_{\omega,\gamma}(\bm{f}) < \mathfrak{n}_{\omega,\gamma}$.
Then, 
\begin{align*}
	K_\gamma(\bm{f})
		\geq 0
	& \ \Longleftrightarrow\ 
	\|\bm{f}\|_{L^2(\mathcal{G})}^{1-s_c}\|\partial_x\bm{f}\|_{L^2(\mathcal{G})}^{s_c}
		< \|Q\|_{L^2(\mathbb{R})}^{1-s_c}\|\partial_xQ\|_{L^2(\mathbb{R})}^{s_c} \\
	& \ \Longleftrightarrow\ 
	\|\bm{f}\|_{L^2(\mathcal{G})}^{1-s_c}\|\bm{f}\|_{\dot{H}_\gamma^1(\mathcal{G})}^{s_c}
		< \|Q\|_{L^2(\mathbb{R})}^{1-s_c}\|\partial_xQ\|_{L^2(\mathbb{R})}^{s_c}, \\
	K_\gamma(\bm{f})
		< 0
	& \ \Longleftrightarrow\ 
	\|\bm{f}\|_{L^2(\mathcal{G})}^{1-s_c}\|\partial_x\bm{f}\|_{L^2(\mathcal{G})}^{s_c}
		> \|Q\|_{L^2(\mathbb{R})}^{1-s_c}\|\partial_xQ\|_{L^2(\mathbb{R})}^{s_c} \\
	& \ \Longleftrightarrow\ 
	\|\bm{f}\|_{L^2(\mathcal{G})}^{1-s_c}\|\bm{f}\|_{\dot{H}_\gamma^1(\mathcal{G})}^{s_c}
		> \|Q\|_{L^2(\mathbb{R})}^{1-s_c}\|\partial_xQ\|_{L^2(\mathbb{R})}^{s_c}.
\end{align*}
\end{proposition}

\begin{proof}
First, we prove that
\begin{align*}
	\|\bm{f}\|_{L^2(\mathcal{G})}^{1-s_c}\|\bm{f}\|_{\dot{H}_\gamma^1(\mathcal{G})}^{s_c}
		\neq \|Q\|_{L^2(\mathbb{R})}^{1-s_c}\|\partial_xQ\|_{L^2(\mathbb{R})}^{s_c}
\end{align*}
if $S_{\omega,\gamma}(\bm{f}) < \mathfrak{n}_{\omega,\gamma}$.
Corollary \ref{cor2.5} and \eqref{G-N inequality} with Lemma \ref{lem2.2} give us
\begin{align*}
	&M^\text{line}(Q)^\frac{1-s_c}{s_c}E_0^\text{line}(Q)
	\\
		& > \|\bm{f}\|_{L^2(\mathcal{G})}^\frac{2(1-s_c)}{s_c}\left(\frac{1}{2}\|\bm{f}\|_{\dot{H}_\gamma^1(\mathcal{G})}^2 - \frac{1}{p+1}\|\bm{f}\|_{L^{p+1}(\mathcal{G})}^{p+1}\right) \\
		& \geq \frac{1}{2}\|\bm{f}\|_{L^2(\mathcal{G})}^\frac{2(1-s_c)}{s_c}\|\bm{f}\|_{\dot{H}_\gamma^1(\mathcal{G})}^2 - \frac{C_\text{GN}^\text{line}}{p+1}\|\bm{f}\|_{L^2(\mathcal{G})}^\frac{(p-1)(1-s_c)}{2s_c}\|\bm{f}\|_{\dot{H}_\gamma^1(\mathcal{G})}^\frac{p-1}{2}
\end{align*}
and hence, we have
\begin{align*}
	1
		> \frac{p-1}{p-5}\cdot\frac{\|\bm{f}\|_{L^2(\mathcal{G})}^\frac{2(1-s_c)}{s_c}\|\bm{f}\|_{\dot{H}_\gamma^1(\mathcal{G})}^2}{\|Q\|_{L^2(\mathbb{R})}^\frac{2(1-s_c)}{s_c}\|\partial_xQ\|_{L^2(\mathbb{R})}^2} - \frac{4}{(p-5)}\cdot\frac{\|\bm{f}\|_{L^2(\mathcal{G})}^\frac{(p-1)(1-s_c)}{2s_c}\|\bm{f}\|_{\dot{H}_\gamma^1(\mathcal{G})}^\frac{p-1}{2}}{\|Q\|_{L^2(\mathbb{R})}^\frac{(p-1)(1-s_c)}{2s_c}\|\partial_xQ\|_{L^2(\mathbb{R})}^\frac{p-1}{2}},
\end{align*}
which implies the desired result.
Replacing $\|\bm{f}\|_{\dot{H}_\gamma^1(\mathcal{G})}$ with $\|\partial_{x}\bm{f}\|_{L^{2}(\mathcal{G})}$ for the above argument, we also have
\begin{align}
\label{eq2.2.0}
	\|\bm{f}\|_{L^2(\mathcal{G})}^{1-s_c}\|\partial_x\bm{f}\|_{L^2(\mathcal{G})}^{s_c}
		\neq \|Q\|_{L^2(\mathbb{R})}^{1-s_c}\|\partial_xQ\|_{L^2(\mathbb{R})}^{s_c}. 
\end{align}

If
\begin{align*}
	\|\bm{f}\|_{L^2(\mathcal{G})}^{1-s_c}\|\partial_x\bm{f}\|_{L^2(\mathcal{G})}^{s_c}
		< \|Q\|_{L^2(\mathbb{R})}^{1-s_c}\|\partial_xQ\|_{L^2(\mathbb{R})}^{s_c},
\end{align*}
then Corollary \ref{cor2.5} and \eqref{G-N inequality} for $\gamma=0$ with Lemma \ref{lem2.2} give
\begin{align*}
	&M^\text{line}(Q)^\frac{1-s_c}{s_c}E_0^\text{line}(Q)
	\\
		& \geq \frac{1}{2}\|\bm{f}\|_{L^2(\mathcal{G})}^\frac{2(1-s_c)}{s_c}\|\bm{f}\|_{\dot{H}_\gamma^1(\mathcal{G})}^2 - \frac{C_\text{GN}^\text{line}}{p+1}\|\bm{f}\|_{L^2(\mathcal{G})}^\frac{(p-1)(1-s_c)}{2s_c}\|\partial_x\bm{f}\|_{L^2(\mathcal{G})}^\frac{p-1}{2} \\
		& > \frac{1}{2}\|\bm{f}\|_{L^2(\mathcal{G})}^\frac{2(1-s_c)}{s_c}\|\bm{f}\|_{\dot{H}_\gamma^1(\mathcal{G})}^2 - \frac{2}{p-1}\|Q\|_{L^2(\mathbb{R})}^\frac{2(1-s_c)}{s_c}\|\partial_xQ\|_{L^2(\mathbb{R})}^2.
\end{align*}
Thus, we have
\begin{align*}
	\frac{1}{2}\|\bm{f}\|_{L^2(\mathcal{G})}^\frac{2(1-s_c)}{s_c}\|\bm{f}\|_{\dot{H}_\gamma^1(\mathcal{G})}^2
		& < M^\text{line}(Q)^\frac{1-s_c}{s_c}E_0^\text{line}(Q) + \frac{2}{p-1}\|Q\|_{L^2(\mathbb{R})}^\frac{2(1-s_c)}{s_c}\|\partial_xQ\|_{L^2(\mathbb{R})}^2 \\
		& = \frac{1}{2}\|Q\|_{L^2(\mathbb{R})}^\frac{2(1-s_c)}{s_c}\|\partial_xQ\|_{L^2(\mathbb{R})}^2.
\end{align*}
If $K_\gamma(\bm{f}) < 0$, then it follows from \eqref{G-N inequality} for $\gamma=0$ with Lemma \ref{lem2.2} that
\begin{align*}
	0
		> K_\gamma(\bm{f})
		\geq 2\|\partial_x\bm{f}\|_{L^2(\mathcal{G})}^2 - \frac{(p-1)C_\text{GN}^\text{line}}{p+1}\|\bm{f}\|_{L^2(\mathcal{G})}^\frac{p+3}{2}\|\partial_x\bm{f}\|_{L^2(\mathcal{G})}^\frac{p-1}{2}.
\end{align*}
This inequality deduces that
\begin{align*}
	\|\bm{f}\|_{L^2(\mathcal{G})}^{1-s_c}\|\partial_x\bm{f}\|_{L^2(\mathcal{G})}^{s_c}
		\geq \left[\frac{2(p+1)}{(p-1)C_\text{GN}^\text{line}}\right]^\frac{1}{p-1}
		= \|Q\|_{L^2(\mathbb{R})}^{1-s_c}\|\partial_xQ\|_{L^2(\mathbb{R})}^{s_c}
\end{align*}
and the equality never occurs by \eqref{eq2.2.0}.

On the other hand, if $K_\gamma(\bm{f}) \geq 0$, then
\begin{align*}
	& \left\{\frac{(p-1)\omega}{p+3}\right\}^\frac{p+3}{2(p-1)}\|\bm{f}\|_{L^2(\mathcal{G})}^\frac{p+3}{p-1}\|\bm{f}\|_{\dot{H}_\gamma^1(\mathcal{G})}^\frac{p-5}{p-1} \\
		& \leq \frac{\omega}{2}\|\bm{f}\|_{L^2(\mathcal{G})}^2 + \frac{p-5}{2(p-1)}\|\bm{f}\|_{\dot{H}_\gamma^1(\mathcal{G})}^2 \\
		& \leq \frac{\omega}{2}\|\bm{f}\|_{L^2(\mathcal{G})}^2 + \frac{p-5}{2(p-1)}\|\bm{f}\|_{\dot{H}_\gamma^1(\mathcal{G})}^2 + \frac{1}{p-1}K_\gamma(\bm{f}) \\
		& \leq S_{\omega,\gamma}(\bm{f})
		< \omega^\frac{p+3}{2(p-1)}S_{1,0}^\text{line}(Q)
		= \omega^\frac{p+3}{2(p-1)}\left(\frac{p-1}{p+3}\right)^\frac{p+3}{2(p-1)}\|Q\|_{L^2(\mathbb{R})}^\frac{p+3}{p-1}\|\partial_xQ\|_{L^2(\mathbb{R})}^\frac{p-5}{p-1}.
\end{align*}
\end{proof}

\begin{remark}
\label{rmk2.7}
The statement ``$\|\bm{f}\|_{L^2(\mathcal{G})}^{1-s_c}\|\partial_x\bm{f}\|_{L^2(\mathcal{G})}^{s_c} < \|Q\|_{L^2(\mathbb{R})}^{1-s_c}\|\partial_xQ\|_{L^2(\mathbb{R})}^{s_c}$ implies $K_{\gamma}(\bm{f})>0$'' is true even without $S_{\omega,\gamma}(\bm{f}) < \mathfrak{n}_{\omega,\gamma}$ since we do not use it in the above proof. We will use this fact in Section \ref{Sec:Sca}. 
\end{remark}

We define potential well sets $PW_{\gamma}^{\pm}$ by
\begin{align*}
	PW_{\gamma}&:=\{\bm{f} \in H_{c}^1(\mathcal{G}) : M(\bm{f})^{\frac{1-s_{c}}{s_{c}}}E_\gamma(\bm{f}) < M^\mathrm{line}(Q)^{\frac{1-s_{c}}{s_{c}}}E_0^\mathrm{line}(Q)\},
	\\
	PW_{\gamma}^{+}&:=\{\bm{f} \in PW_{\gamma} : \| \bm{f}\|_{L^{2}(\mathcal{G})}^{1-s_c}\| \bm{f} \|_{\dot{H}_{\gamma}^{1}(\mathcal{G})}^{s_c} < \|Q\|_{L^{2}(\mathbb{R})}^{1-s_c}\|Q\|_{\dot{H}^{1}(\mathbb{R})}^{s_c}\},
	\\
	PW_{\gamma}^{-}&:=\{\bm{f} \in PW_{\gamma} : \| \bm{f}\|_{L^{2}(\mathcal{G})}^{1-s_c}\| \bm{f}\|_{\dot{H}_{\gamma}^{1}(\mathcal{G})}^{s_c} > \|Q\|_{L^{2}(\mathbb{R})}^{1-s_c}\|Q\|_{\dot{H}^{1}(\mathbb{R})}^{s_c}\}. 
\end{align*}

\begin{lemma}[Coercivity]\label{Coercivity}
Let $p > 5$, $\mu = -1$, and $\gamma \geq 0$.
Let $\bm{u}$ be a solution to \eqref{NLS} with initial data $\bm{u_0}$.
\begin{itemize}
\item ($PW_{\gamma}^{+}$ case)
If $\bm{u_0} \in PW_{\gamma}^{+}$, then $\bm{u}(t) \in PW_{\gamma}^{+}$ for each $t \in (T_\text{min},T_\text{max})$ and satisfies
\begin{align*}
	K_\gamma(\bm{u}(t))
		\geq \min\{\mathfrak{n}_{\omega,\gamma}-S_{\omega,\gamma}(\bm{u_0}),c(p)\|\bm{u}(t)\|_{\dot{H}_{\gamma}^{1}(\mathcal{G})}^2\}
\end{align*}
for some positive constant $c(p)$ and some $\omega>0$. 
\item ($PW_{\gamma}^{-}$ case)
If $\bm{u_0} \in PW_{\gamma}^{-}$, then $\bm{u}(t) \in PW_{\gamma}^{-}$ for each $t \in (T_\text{min},T_\text{max})$ and satisfies
\begin{align*}
	K_\gamma(\bm{u}(t))
		< 4(S_{\omega,\gamma}(\bm{u_0})-\mathfrak{n}_{\omega,\gamma})
		< 0
\end{align*}
for some $\omega>0$.
\end{itemize}
\end{lemma}

\begin{proof}
The proof is similar to the line case. See \cite[Proposition 2.18]{IkeInu17}.
\end{proof}

\begin{lemma}\label{Equivalence of S and H1}
Let $p > 5$, $\mu = -1$, $\gamma \geq 0$, and $\omega > 0$.
If $\bm{f} \in H_c^1(\mathcal{G})$ satisfies $K_{\gamma}(\bm{f}) \geq 0$, then 
\begin{align*}
	S_{\omega,\gamma}(\bm{f})
		\leq \frac{\omega}{2}\|\bm{f}\|_{L^2(\mathcal{G})}^2 + \frac{1}{2}\|\partial_x\bm{f}\|_{L^2(\mathcal{G})}^2 + \frac{3\gamma}{2}|f_1(0)|^2
		\leq \frac{p-1}{p-5}S_{\omega,\gamma}(\bm{f}).
\end{align*}
\end{lemma}

\begin{proof}
This follows from direct calculations.
\end{proof}

\begin{corollary}[Global well-posedness]\label{Global well-posedness}
Let $p > 5$, $\mu = -1$, and $\gamma \geq 0$.
If $\bm{u_0} \in PW_{\gamma}^{+}$, then a solution $\bm{u}$ to \eqref{NLS} with initial data $\bm{u_0}$ exists globally in time. Moreover,   if $\bm{u}_{0} \neq 0$, then there exists a positive constant $\delta$ such that
\begin{align*}
	K_{\gamma}(\bm{u}(t)) \geq \delta
\end{align*}
for all $t \in \mathbb{R}$. 
\end{corollary}

\begin{proof}
By Lemma \ref{Coercivity}, $\bm{u}(t) \in PW_{\gamma}^{+}$ for each $t \in (T_\text{min},T_\text{max})$.
Then, Lemmas \ref{Coercivity} and \ref{Equivalence of S and H1} imply the desired result.
\end{proof}


\section{Preliminaries for the proof of scattering}\label{Sec.Pre}
\subsection{Strichartz estimates}

Let $q'$ denote the H\"{o}lder conjugte of $q \in [1,\infty]$. 
 
\begin{theorem}[Dispersive estimate, \cite{GreIgn19}]\label{Dispersive estimate}
Let $\gamma \geq 0$ and $q \in [1,2]$.
Then, 
\begin{align*}
	\|e^{it\Delta_\mathcal{G}^\gamma}\bm{f}\|_{L_x^{q'}}
		\lesssim |t|^{\frac{1}{2} - \frac{1}{q}}\|\bm{f}\|_{L_x^q}, \quad (t\in\mathbb{R}\setminus\{0\}). 
\end{align*}
\end{theorem}

\begin{definition}[$L^2$-admissible]
We say that a pair $(q_{0},r_{0})$ is $L^2$-admissible if $(q_{0},r_{0})$ satisfies $2 \leq q_{0}, r_{0} \leq \infty$ and
\begin{align*}
	\frac{2}{q_{0}} + \frac{1}{r_{0}}
		= \frac{1}{2}.
\end{align*}
\end{definition}

We define time-space norms by
\begin{align*}
	\| \bm{v}\|_{L^{q}(I;X)}
	:=\|  \| \bm{v}\|_{X} \|_{L^{q}(I)}
\end{align*}
for a Banach space $X$ consisting of functions on $\mathcal{G}$. 

\begin{theorem}[Strichartz estimates for addmissible pair, \cite{GreIgn19}]\label{Strichartz estimates for non-admissible pair}
Let $(q_1,r_1)$ and $(q_2,r_2)$ be $L^2$-admissible.
Then, 
\begin{align}
	\|e^{it\Delta_\mathcal{G}^\gamma}\bm{f}\|_{L_t^{q_1}L_x^{r_1}}
		\lesssim \|\bm{f}\|_{L_x^2}. \label{102}
\end{align}
If $I$ is a time interval and $t_0 \in \overline{I}$, then
\begin{align}
	\left\|\int_{t_0}^te^{i(t-s)\Delta_\mathcal{G}^\gamma}\bm{F}(s,\cdot)ds\right\|_{L_t^{q_1}(I;L_x^{r_1})}
		\lesssim \|\bm{F}\|_{L_t^{q_2'}(I;L_x^{r_2'})}. \label{103}
\end{align}
\end{theorem}

For $p>5$, we define
\begin{align*}
	r:=p+1, \quad 
	a:=\frac{2(p^{2}-1)}{p+3}, \quad 
	b:=\frac{2(p^{2}-1)}{p^{2}-3p-2}.
\end{align*}

\begin{theorem}[Strichartz estimates for non-admissible pair]\label{Strichartz estimates for non-admissible pair}
Let $p > 5$.
Let $I$ be a time interval and $t_0 \in \overline{I}$.
Then, the following inequalities hold:
\begin{align*}
	&\|e^{it\Delta_\mathcal{G}^\gamma}\bm{f}\|_{L_t^a L_x^{r} \cap L_t^{p-1}L_x^\infty}
		\lesssim \|\bm{f}\|_{H^1},
	\\
	&\left\|\int_{t_0}^te^{i(t-s)\Delta_\mathcal{G}^\gamma}\bm{F}(s,\cdot)ds\right\|_{L_t^a(I;L_x^{r}) \cap L_t^{p-1}(I;L_x^\infty)}
		\lesssim \|\bm{F}\|_{L_t^{b'}(I;L_x^{r'})}.
\end{align*}
\end{theorem}

\begin{proof}
The proof is the same as in the line case. See e.g. \cite{Fos05} and \cite{BanVis16}.
\end{proof}

\subsection{Long time perturbation and existence of the wave operator}

\begin{proposition}[Long time perturbation]
\label{LTP}
Let $I$ be a time interval with $t_0 \in I$ and $M > 0$.
Let $\bm{v} \in C(I;H_{c}^1) \cap L_t^a(I;L_x^r)$ satisfy
\begin{align*}
	\bm{v}(t,x)
		= e^{i(t-t_0)\Delta_\mathcal{G}^\gamma}\bm{v}(0,x) + i\int_{t_0}^t e^{i(t-s)\Delta_\mathcal{G}^\gamma}(|\bm{v}|^{p-1}\bm{v})(s,x)ds + \bm{e}(t,x)
\end{align*}
for some function $\bm{e}$ and $\|\bm{v}\|_{L_t^a(I;L_x^r)} \leq M$.
Let $\bm{u}$ be a solution to \eqref{NLS}.
There exist $\varepsilon_0 = \varepsilon_0(M) > 0$ and $C = C(M) > 0$ such that if
\begin{align*}
	\|e^{i(t-t_0)\Delta_\mathcal{G}^\gamma}(\bm{u}(t_0)-\bm{v}(t_0))\|_{L_t^a(I;L_x^r)} + \|\bm{e}\|_{L_t^a(I;L_x^r)}
		< \varepsilon
\end{align*}
for some $0 \leq \varepsilon \leq \varepsilon_0$, then the maximal existence time interval of $\bm{u}$ contains $I$ and
\begin{align*}
	\|\bm{u}-\bm{v}\|_{L_t^a(I;L_x^r)}
		\leq C\varepsilon.
\end{align*}
\end{proposition}

\begin{proof}
We can show this in the similar way to the line case. See e.g. \cite{FanXieCaz11}. 
\end{proof}

\begin{corollary}[Small data scattering]
\label{Small data scattering}
Let $p > 5$, $\gamma \geq 0$, $t_0 > 0$, and $\bm{u_0} \in H_c^1(\mathcal{G})$.
Then, there exists $\varepsilon_0 > 0$ such that for any $0 < \varepsilon \leq \varepsilon_0$, if	
\begin{align*}
	\|e^{i(t-t_0)\Delta_\mathcal{G}^\gamma}\bm{u}_0\|_{L_t^a(t_0,\infty;L_x^r)}
		\leq \varepsilon,
\end{align*}
then the solution $\bm{u}$ to \eqref{NLS} with initial data $u(t_0) = u_0$ exists globally forward in time and satisfies
\begin{align*}
	\|\bm{u}\|_{L_t^a(t_0,\infty;L_x^r)}
		\lesssim \varepsilon.
\end{align*}
Moreover, $\bm{u}$ scatters.
\end{corollary}

\begin{proof}
Applying Proposition \ref{LTP} as $I=(t_{0},\infty)$, $\bm{v}=0$, and $\bm{e}=0$, we get the statement. 
\end{proof}

\begin{proposition}[Existence of the wave operator]
\label{WOp}
Let $p > 5$ and $\gamma \geq 0$.
Suppose that $\bm{u}_+ \in H_c^1(\mathcal{G})$ satisfies
\begin{align*}
	\frac{\omega}{2}M(\bm{u}_+) + L_{\gamma}(\bm{u}_+)
		< \mathfrak{n}_{\omega,\gamma}\ \text{ for some }\ \omega > 0,
\end{align*}
where $L_{\gamma}$ is defined as $L_{\gamma}(\bm{f}) := \frac{1}{2}\|\partial_x\bm{f}\|_{L^2}^2 + \frac{3\gamma}{2}|f_1(0)|^2$.
Then, there exists $\bm{u_0} \in H_c^1(\mathcal{G})$ such that the solution $\bm{u}$ to \eqref{NLS} satisfies $\bm{u_0} \in PW_{\gamma}^{+}$,
\begin{align*}
	\|\bm{u}\|_{L_t^a(0,\infty;L_x^{r})}
		\leq 2\|e^{it\Delta_\mathcal{G}^\gamma}\bm{u_0}\|_{L_t^a(0,\infty;L_x^{r})},
\end{align*}
and
\begin{align*}
	\lim_{t \rightarrow + \infty}\|\bm{u}(t) - e^{it\Delta_\mathcal{G}^\gamma}\bm{u}_+\|_{H^1}
		= 0.
\end{align*}
\end{proposition}

\begin{proof}
The proof is similar to the line case. See e.g. \cite[Proposition 4.6]{FanXieCaz11}. 
\end{proof}
\subsection{Localized virial identity}

For $R > 0$, let $\mathscr{X}_R\in C_0^\infty((0,\infty))$ be a cut-off function with
\begin{equation}\label{002}
	\mathscr{X}_R(x)
		= R^2\mathscr{X}\left(\frac{x}{R}\right),\ \text{ where }\ 
	\mathscr{X}(x)
		:=
		\left\{
		\begin{array}{cl}
		\hspace{-0.2cm}x^2 & (0 \leq x \leq 1),\\
		\hspace{-0.2cm}\text{smooth} & (1 \leq x \leq 3), \\
		\hspace{-0.2cm}0 & (3 \leq x)
		\end{array}
		\right.
\end{equation}
and $\mathscr{X}''(x) \leq 2$ for $x>0$.

\begin{lemma}[Localized virial identities, \cite{GolOht20}]\label{Generalized virial identities}
Let $p > 1$, $\gamma \geq 0$, and $\bm{u} \in C((T_\text{min},T_\text{max});H_c^1(\mathcal{G}))$ be a solution to \eqref{NLS}.
For $R>0$, let
\begin{align*}
	V(t)
		:= \int_\mathcal{G} \mathscr{X}_R(x)|\bm{u}(t,x)|^2dx.
\end{align*}
Then, we have
\begin{align*}
	V'(t)
		& = 2\operatorname{Im}\int_\mathcal{G} \mathscr{X}_R'(x)\overline{\bm{u}(t,x)}\partial_x \bm{u}(t,x)dx, \\
	V''(t)
		& = 4\int_\mathcal{G} \mathscr{X}_R''(x)|\partial_x\bm{u}(t,x)|^2dx+6\gamma \mathscr{X}_R''(0)|u_1(t,0)|^2 \\
		& \hspace{2.0cm} - \int_\mathcal{G} \mathscr{X}_R^{(4)}(x)|\bm{u}(t,x)|^2dx - \frac{2(p-1)}{p+1}\int_\mathcal{G} \mathscr{X}_R''(x)|\bm{u}(t,x)|^{p+1}dx.
\end{align*}
\end{lemma}

\subsection{Technical lemma}

We denote the Laplacian with the delta potential on the real line by $\Delta_\mathbb{R}^\gamma$. More precisely, it is defined by
\begin{align*}
	\mathcal{D}(\Delta_\mathbb{R}^\gamma) &:= 
	\left\{ f \in H^{2}(\mathbb{R}) : 
	f'(0+)-f'(0-)=2\gamma f(0) \right\},
	\\
	\Delta_\mathbb{R}^\gamma f&:=\partial_{xx} f.
\end{align*}

\begin{lem}[\cite{BanVis16}]\label{Lem:Hgam}
Let $\gam\ge 0$. Define $\tau_yf(x)=f(x-y)$. 
Let $\{y_n\}_{n=1}^\infty\subset (0,\infty)$ be a sequence with $y_n\to \infty$ or $-\infty$. 
Then the following hold: 
\begin{enumerate}[(i)]
\item For $f\in L^2(\mathbb{R})$,
$$
e^{it\Dgr} \tau_{y_n} f - e^{it\rd_{xx} }\tau_{y_n} f\xrightarrow{n\to\infty} 0 \qquad \text{in } L^a_tL^r_x.
$$
\item For $f\in L^2(\mathbb{R})$ and $\{t_n\}_{n=1}^\infty \subset \R$, 
$$
e^{-i t_n \Dgr} \tau_{y_n} f - e^{-i t_n \rd_{xx}} \tau_{y_n} f
\xrightarrow{n\to\infty} 0 \qquad \text{in } H^1(\R).
$$
\item For $h\in L^{b'}_tL^{r'}_x$,
$$
\int_0^t e^{i(t-s) \Dgr} \tau_{y_n} h (s) ds - \int_0^t e^{i(t-s) \rd_{xx}} \tau_{y_n} h (s) ds
\xrightarrow{n\to\infty} 0 \qquad \text{in } L^{a}_tL^{r}_x.
$$
\end{enumerate}
\end{lem}

\section{Linear profile decomposition}\label{Sec:Lin}




\subsection{Odd-even decomposition}

Let $\bvph\in L^2 (\boG)$. Define $\al_k(\bvph):(0,\infty)\to \mathbb{C}$ ($k=1,2,3$) by
\begin{align*}
\al_1(\bvph) := \frac{-2\vph_1 + \vph_2 + \vph_3}3, 
&&\al_2(\bvph) := \frac{-\vph_1 - \vph_2 +2 \vph_3}3, 
&&\al_3(\bvph) := \frac{\vph_1 + \vph_2 + \vph_3}3.
\end{align*}
Note that if $\bvph\in \boD(\Dgg)$, we have
\begin{equation}\label{Condi:Origin}
\al_1(\bvph)(0) = \al_2(\bvph)(0)= 0,\qquad \al_3(\bvph)'(0+)=\gam  \al_3 (0).
\end{equation}
Then, we define $\mathcal{Q}_k: L^2(\boG)\to L^2(\mathbb{R})$ by
$$
\mathcal{Q}_1 \bm{\varphi}(x) := 
\left\{
\begin{aligned}
&-\al_1(\bm{\varphi}) (-x) & (x<0)\\
&\al_1(\bm{\varphi}) (x) & (x\geq0)
\end{aligned}
\right.
,
\qquad
\mathcal{Q}_2 \bm{\varphi}(x) := 
\left\{
\begin{aligned}
&-\al_2(\bm{\varphi}) (-x) & (x<0)\\
&\al_2(\bm{\varphi}) (x) & (x\geq0)
\end{aligned}
\right.
,
$$
$$
\mathcal{Q}_3 \bm{\varphi} (x) := 
\left\{
\begin{aligned}
&\al_3(\bm{\varphi}) (-x) & (x<0)\\
&\al_3(\bm{\varphi}) (x) & (x\geq0)
\end{aligned}
\right.
.
$$
Clearly, each $\mathcal{Q}_k$ is a bounded operator. 
We now see that $\mathcal{Q}_k$ enjoys commutability of the Laplacians as follows. 


\begin{lem}\label{Lem:Comm:Diff}
For each $k=1,2,3$, $\mathcal{Q}_j (\boD (\Dgg)) \subset \boD(\Del_\R^\gam)$, and
\begin{equation}\label{Eq:Comm:Diff}
\mathcal{Q}_k \Dgg \bvph = \Del_\R^\gam \mathcal{Q}_k \bvph\qquad (\bvph\in \boD (\Dgg)).
\end{equation}
\end{lem}

\begin{proof}
Let $\bvph\in \boD(\Dgg)$. 
By definition we have $\mathcal{Q}_k\bvph \in H^2(\R\backslash \{0\})$. 
From \eqref{Condi:Origin}, it follows that
\begin{align*}
(\mathcal{Q}_k\bvph)(0+)= (\mathcal{Q}_k\bvph)(0-) = 0,&& 
(\mathcal{Q}_k\bvph)'(0+) - (\mathcal{Q}_k\bvph)'(0-) =0,&& (k=1,2),
\end{align*}
\begin{align*}
(\mathcal{Q}_3\bvph)(0+)= (\mathcal{Q}_3\bvph)(0-) = \al_3(0),&& 
(\mathcal{Q}_3\bvph)'(0+) - (\mathcal{Q}_3\bvph)'(0-) =2\gam \al_3(0)
\end{align*}
and hence $\mathcal{Q}_k\bvph\in \boD(\Delta_{\mathbb{R}}^{\gamma})$. 
We also have
$$
\begin{aligned}
\mathcal{Q}_1 \Dgg \bvph (x) &= 
\left\{
\begin{aligned}
&-\al_1(\Dgg \bvph) (-x) & (x<0)\\
&\al_1(\Dgg \bvph) (x) & (x>0)
\end{aligned}
\right.
=
\left\{
\begin{aligned}
&-\frac{-2\vph''_1 + \vph''_2 + \vph''_3}3  (-x) & (x<0)\\
&\frac{-2\vph''_1 + \vph''_2 + \vph''_3}3  (x) & (x>0)
\end{aligned}
\right.\\
&=\Del_\R^\gam \mathcal{Q}_1 \bm{\varphi}(x).
\end{aligned}
$$
The cases $k=2,3$ follows similarly.
\end{proof}

As a consequence, we can show that 
the Schr\"odinger propagators are commutable.

\begin{prop}\label{Prop:Comm:Sch}
For each $k=1,2,3$, 
\begin{equation}\label{b2.2}
\mathcal{Q}_k e^{it\Dgg} \bvph = e^{it\Del_\R^\gam} \mathcal{Q}_k \bvph\qquad (\bvph\in L^2(\boG)).
\end{equation}
\end{prop}

\begin{proof}
We consider the Yosida approximation: For $\la>0$, we define
\begin{align*}
J_{\boG, \la} := (I - i\la \Dgg)^{-1} ,&&
J_{\la} := (I - i\la \Delta_{\mathbb{R}}^{\gamma})^{-1}.
\end{align*}
Then it is known that
\begin{align*}
e^{t\frac{J_{\boG, \la}-I}{\la}} \bvph \to e^{it\Dgg} \bvph,&&
e^{t\frac{J_{\la}-I}{\la}} f \to e^{it\Del_\R^\gam} f
\end{align*}
for all $t\in\R$, $\bvph\in L^2(\boG)$ and $f\in L^2(\R)$. (See \cite[Theorem 3.1]{CazHar98} for example.) 
Now Lemma \ref{Lem:Comm:Diff} gives
$$
\mathcal{Q}_k J_{\boG, \la} = J_{\la}\mathcal{Q}_k,\qquad (k=1,2,3).
$$
Hence it follows that
$$
\mathcal{Q}_k e^{t\frac{J_{\boG, \la}-I}{\la}} \bvph = e^{t\frac{J_{\la}-I}{\la}} \mathcal{Q}_k \bvph
$$
for all $\bvph\in L^2(\boG)$ and $k=1,2,3$. Letting $\la\to 0$, we obtain the conclusion.
\end{proof}

By Proposition \ref{Prop:Comm:Sch}, the linear Schr\"odinger evolution on $\boG$ can be completely described by that on $\R$ via $\mathcal{Q}_k$. 
To make it more explicit, we introduce the space 
$\boK:=(L^2_{\odd}(\R))^2\times (L^2_{\even}(\R))$, and define 
\begin{align}\label{Eq:Qinv0}
\mathcal{Q}: L^2(\boG) \to \boK,&&
\mathcal{Q}\bvph := 
\begin{bmatrix}
\mathcal{Q}_1\bvph\\
\mathcal{Q}_2\bvph\\
\mathcal{Q}_3\bvph
\end{bmatrix},
\end{align}
where we intentionally use square bracket for elements of $\boK$ to distinguish it from functions on graph. 
Then $\mathcal{Q}$ is invertible, and
\begin{align}\label{Eq:Qinv}
\mathcal{Q}^{-1} 
\begin{bmatrix}
f_1\\
f_2\\
f_3
\end{bmatrix}
=
\begin{pmatrix}
-1 & 0 & 1\\
1 & -1 & 1\\
0 & 1 & 1
\end{pmatrix}
\begin{bmatrix}
f_1\\
f_2\\
f_3
\end{bmatrix}
\bcp,&&
\begin{bmatrix}
f_1\\
f_2\\
f_3
\end{bmatrix}
\in \boK
.
\end{align}
Therefore, Proposition \ref{Prop:Comm:Sch} gives the following formula:
\begin{equation}
e^{it\Dgg} \bvph =
\mathcal{Q}^{-1} e^{it\Del_\R^\gam} \mathcal{Q} \bvph.
\end{equation}

\begin{remark}
For the star graph with $N$ edges, the space $\boK$ is changed into $(L^2_{\odd}(\R))^{N-1}\times L^2_{\even} (\R)$, and $\mathcal{Q}$ is defined 
by \eqref{Eq:Qinv0} where the matrix of \eqref{Eq:Qinv} is replaced by
$$
\begin{pmatrix}
-1 & 0 & 0 & \cdots & 0 & 1 \\
1 & -1 & 0 & \cdots & 0 & 1 \\
0 & 1 & -1 & \cdots & 0 & 1 \\
&&&\vdots &&\\
0 & 0 & 0 & \cdots & -1 & 1\\
0 & 0 & 0 & \cdots & 1 & 1
\end{pmatrix}
.
$$
\end{remark}

\subsection{Linear profile decomposition}

The main claim of this section is the following.

\begin{prop}[Linear profile decomposition]\label{Prop:LPD}
Let $\{ \bvph_n\}_{n=1}^\infty$ be a bounded sequence in $H^1_c (\boG)$. Then there exist
$$
t_n^j\in \R,\quad y_n^j \ge 0,\quad
 \psi_k^j\in H^1(\R),\qquad j\ge 1,\  k=1,2,3,\  n\in\N
$$
such that, by taking subsequence if necessary, we have
\begin{equation}\label{Eq:LPD}
\bvph_n = \sum_{j=0}^J 
\boT_{t_n^j, y_n^j} (\psi_1^j,\psi_2^j,\psi_3^j)
+ \bow_n^J,\qquad \forall J\in\N,
\end{equation}
\begin{equation}\label{shift}
\begin{aligned}
\boT_{t ,y} (\psi_1,\psi_2,\psi_3) :=&
e^{-it\Dgg} 
\mathcal{Q}^{-1}
\begin{bmatrix}
\tau_y \psi_1 - \tau_{-y} \boR \psi_1 \\
\tau_y \psi_2 - \tau_{-y} \boR \psi_2 \\
\tau_y \psi_3 + \tau_{-y} \boR \psi_3
\end{bmatrix}
\\
=&
\begin{pmatrix}
e^{-it\Dgr}\tau_{y} \psi_1 \mathbbm{1}_{>0} 
+ \frac 13 e^{-it\Dgr}\tau_{-y} \boR(-\psi_1+2\psi_2+2\psi_3) \mathbbm{1}_{>0}\\
e^{-it\Dgr}\tau_{y} \psi_2 \mathbbm{1}_{>0} 
+ \frac 13 e^{-it\Dgr}\tau_{-y} \boR(2\psi_1 -\psi_2+2\psi_3) \mathbbm{1}_{>0}\\
e^{-it\Dgr}\tau_{y} \psi_3 \mathbbm{1}_{>0} 
+ \frac 13 e^{-it\Dgr}\tau_{-y} \boR (2\psi_1+2\psi_2 -\psi_3) \mathbbm{1}_{>0}\\
\end{pmatrix}
\\
&\hspace{110pt} 
(t,y\in\R,\quad \psi_1,\psi_2,\psi_3\in H^1(\R))
\end{aligned}
\end{equation}
with $\boR f (\cdot):= f(-\cdot)$ for $f:\R\to\C$, where all of the following hold:
\begin{enumerate}[(i)]
\item For each $j\in\N$, 
\begin{align*}
&\text{either}\quad t_n^j=0\quad (n\in\N), &\text{or}& 
&t_n^j\to \pm\infty\quad
 (n\to \infty),
\\
&\text{either}\quad y_n^j=0\quad (n\in\N),\quad &\text{or}& 
&y_n^j\to\infty\quad (n\to \infty);
\end{align*}
\item Pairwise asymptotic orthogonality of parameters:
$$
|t_n^j -t_n^k| + |y_n^j-y_n^k|
 \to \infty,\quad \forall j\neq k;
$$
\item Smallness of remainder:
$$
\lim_{J\to\infty} \limsup_{n\to\infty} \nor{e^{it\Dgg} \bow_n^J}{L^\infty_{t,x}} = 0;
$$
\item Asymptotic orthogonality of norms: 
For $q\in [2,\infty)$, 
$$
\nor{\bvph_n}{L^q(\boG)}^q = \sum_{j=1}^J \nor{ \boT_{t_n^j, y_n^j} 
(\psi_1^j,\psi_2^j,\psi_3^j)
}{L^q(\boG)}^q 
+ \nor{\bow_n^J}{L^q(\boG)}^q + o_n(1), \qquad \forall J\in\N,
$$
$$
\nor{\bvph_n}{\dot{H}^1_\gam (\boG)}^2 = \sum_{j=1}^J \nor{ \boT_{t_n^j, y_n^j} 
(\psi_1^j,\psi_2^j,\psi_3^j)
}{\dot{H}^1_\gam(\boG)}^2 
+ \nor{\bow_n^J}{\dot{H}^1_\gam (\boG)}^2 + o_n(1), \qquad \forall J\in\N.
$$
\end{enumerate}
\end{prop}

\begin{remark}
%
The definition of profiles in \eqref{shift} is considered natural in the following sense. 
%
%
We can decompose as
$$
\boT_{0,y}(\psi_1,\psi_2,\psi_3) 
= 
\La_{1,y} \psi_1 
+\La_{2,y} \psi_2
+\La_{3,y} \psi_3,
$$
where
\begin{align*}
&\La_{1,y} \psi := 
\begin{pmatrix}
\tau_y \psi \mathbbm{1}_{>0} \\
0\\
0\\
\end{pmatrix}
+
\begin{pmatrix}
-\frac 13 \tau_{-y} \boR\psi \mathbbm{1}_{>0}\\
\frac 23 \tau_{-y} \boR\psi \mathbbm{1}_{>0}\\
\frac 23 \tau_{-y}  \boR\psi \mathbbm{1}_{>0}\\
\end{pmatrix}
,
\\
&\La_{2,y} \psi :=  
\begin{pmatrix}
0 \\
\tau_y \psi \mathbbm{1}_{>0}\\
0\\
\end{pmatrix}
+
\begin{pmatrix}
\frac 23 \tau_{-y} \boR\psi \mathbbm{1}_{>0}\\
-\frac 13 \tau_{-y} \boR\psi \mathbbm{1}_{>0}\\
\frac 23 \tau_{-y}  \boR\psi \mathbbm{1}_{>0}\\
\end{pmatrix}
,
\\
&\La_{3,y} \psi := 
\begin{pmatrix}
0 \\
0 \\
\tau_y \psi \mathbbm{1}_{>0}\\
\end{pmatrix}
+
\begin{pmatrix}
\frac 23 \tau_{-y} \boR\psi \mathbbm{1}_{>0}\\
\frac 23 \tau_{-y} \boR\psi \mathbbm{1}_{>0}\\
-\frac 13 \tau_{-y}  \boR\psi \mathbbm{1}_{>0}\\
\end{pmatrix}
.
\end{align*}
Then $\La_{k,y} \psi\in H^1_c(\boG)$ for each $k=1,2,3$. 
Moreover, if $y\to\infty$, 
$\La_{k,y} \psi$ approaches the function $\psi$ lying on the $k$-th edge of $\boG$. 
\end{remark}

%



\subsection{Proof of Proposition \ref{Prop:LPD}}

Let $\{\bvph_n\}_{n=1}^\infty\subset H^1_c(\boG)$ be a bounded sequence. 
Then $\{\mathcal{Q}_k \bvph_n\}_{n=1}^\infty$ is also bounded in $H^1(\R)$, and thus 
we can apply profile decomposition for $e^{it \Delta_{\mathbb{R}}^{\gamma}}$. 
Since $\mathcal{Q}_k \bvph_n$ is odd for $k=1,2$ and even for $k=3$, 
each profile can be taken to be odd and even respectively, as shown in the next lemma.


\begin{lem}[Odd/even profile decomposition]\label{Lem:LPD:OE}
Let $\{\bvph_n\}_{n=1}^\infty$ be a bounded sequence of $H^1_c (\boG)$ 
Then, there exist
$$
t_n^j\in \R,\quad y_n^j \ge 0,\quad
 \psi_k^{j}\in H^1(\R),\qquad j\ge 1,\  k=1,2,3,\  n\in\N
$$
such that, by taking subsequence if necessary,
\begin{equation}\label{Eq:LPD:OE}
\mathcal{Q}_k \bvph_n = \sum_{j=1}^J e^{-it_n^j \Dgr} \left( 
\tau_{y_n^j} \psi_k^j \mp \tau_{-y_n^j} \boR \psi_k^j
\right)
 + w_{k,n}^J \mp \boR w_{k,n}^J,\qquad  \forall J\in\N
\end{equation}
where $-$ sign holds for $k=1,2$ 
while $+$ sign for $k=3$, 
and where all of the following hold:
\begin{enumerate}[(i)]
\item For any $j\in\N$,
\begin{align*}
&\text{either}\quad t_n^j=0\quad (n\in\N), &\text{or}& 
&t_n^j\to \pm\infty\quad
 (n\to \infty),
\\
&\text{either}\quad y_n^j=0\quad (n\in\N),\quad &\text{or}& 
&y_n^j\to\infty\quad (n\to \infty);
\end{align*}
\item Pairwise asymptotic orthogonality of parameters:
$$
|t_n^j -t_n^{j'}| + |y_n^j-y_n^{j'}|
 \to \infty,\quad \forall j\neq j';
$$
\item Smallness of remainder:
\begin{equation}\label{Small:w1}
\lim_{J\to\infty} \limsup_{n\to\infty} \nor{e^{it \Dgr} w_{k,n}^J}{L^\infty_tL^\infty_x} = 0,
\end{equation}
\begin{equation}\label{Small:w2}
\tau_{\pm y_n^j} e^{it_n^j \Dgr} w_{k,n}^J\rightharpoonup 0\quad \text{in } H^1(\R) \qquad 
\forall J\in\N,\ 1\le j\le J;
\end{equation}
\item Asymptotic orthogonality of norms: For $q\in [2,\infty)$, 
\begin{equation}\label{Eq:Orth:OM1}
\nor{\mathcal{Q}_k \bvph_n}{L^q(\R)}^2 = \sum_{j=1}^J  \nor{\tau_{y_n^j} \psi_k^j \mp \tau_{-y_n^j} \boR \psi_k^j}{L^q(\R)}^q 
+ \nor{w_n^J \mp \boR w_n^J}{L^q(\R)}^q + o_n(1), \qquad \forall J\in\N,
\end{equation}
\begin{equation}\label{Eq:Orth:OM2}
\nor{\mathcal{Q}_k \bvph_n}{\dot{H}^1_\gam(\R)}^2 = \sum_{j=1}^J \nor{\tau_{y_n^j} \psi_k^j \mp \tau_{-y_n^j} \boR \psi_k^j}{\dot{H}^1_\gam(\R)}^2 
+ \nor{w_n^J \mp \boR w_n^J}{\dot{H}^1_\gam(\R)}^2 + o_n(1), \qquad \forall J\in\N.
\end{equation}
where $-$ sign holds for $k=1,2$ 
while $+$ sign for $k=3$, and $\|f\|_{\dot{H}^1_\gam(\R)}^{2}:=\|\partial_{x}f\|_{L^{2}}^{2} +2\gamma |f(0)|^{2}$. 
\end{enumerate}
\end{lem}

\begin{proof}
The profile decompositions for $e^{it\Dgr}$ under even and odd restrictions follow 
from 
\cite{IkeInu17} and \cite{Inu17}, respectively. 
Here the sequence of parameter can be taken uniformly in $k$ by considering 
$e^{it\Delta_{\mathbb{R}}^{\gamma}}$ as an operator on $\boK$ and by extracting profiles as an element of $\boK$, the trick of which can be seen in the study of NLS system (see \cite{FarPas17,Ham18pre,InKiNi19}).
\end{proof}

Taking $\mathcal{Q}^{-1}$ on \eqref{Eq:LPD:OE}, we obtain
\begin{equation*}
\bvph_n =
\sum_{j=1}^J
\bF_n^j
+ \bow_n^J,
\end{equation*}
where
\begin{align*}
\bF_n^j := \mathcal{Q}^{-1} 
\begin{bmatrix}
e^{-it_n^{j} \Dgr} \left( 
\tau_{y^n_{j}} \psi_{1}^j - \tau_{-y_n^{j}} \boR \psi_{1}^j
\right)
 \\
e^{-it_n^{j} \Dgr} \left( 
\tau_{y^n_{j}} \psi_{2}^j - \tau_{-y_n^{j}} \boR \psi_{2}^j
\right)
 \\
e^{-it_n^{j} \Dgr} \left( 
\tau_{y^n_{j}} \psi_{3}^j + \tau_{-y_n^{j}} \boR \psi_{3}^j
\right)
\end{bmatrix}
,&&
\bow_n^J :=
\mathcal{Q}^{-1} 
\begin{bmatrix}
w_{1,n}^{J} - \boR w_{1,n}^J  \\
w_{2,n}^{J} - \boR w_{2,n}^J \\
w_{3,n}^{J} + \boR w_{3,n}^J 
\end{bmatrix}
.
\end{align*}
By Proposition \ref{Prop:Comm:Sch}, we have
$$
\bF_n^j 
= 
e^{it_n^j\Dgg} \mathcal{Q}^{-1} 
\begin{bmatrix}
\tau_{y_n^j} \psi_1^j - \boR \tau_{y_n^j} \psi_1^j  \\
\tau_{y_n^j} \psi_2^j - \boR \tau_{y_n^j} \psi_2^j \\
\tau_{y_n^j} \psi_3^j + \boR \tau_{y_n^j} \psi_3^j
\end{bmatrix}
= \boT_{t_n^j , y_n^j} (\psi_1^j,\psi_2^j,\psi_3^j).
$$
Thus we obtain the decomposition of the desired form \eqref{Eq:LPD}. \par
%
%
%
%
Hence it suffices to show the conditions (i)-(iv) in the statement. 
(i) and (ii) follow from Lemma \ref{Lem:LPD:OE}. 
By Proposition \ref{Prop:Comm:Sch}, we have
\begin{align*}
e^{it\Dgg} \bow_n^J &=\mathcal{Q}^{-1} e^{it\Dgr} \mathcal{Q} \bow^J_n 
=
\begin{pmatrix}
-1 & 0 & 1\\
1 & -1 & 1\\
0 & 1 & 1
\end{pmatrix}
\begin{bmatrix}
e^{it\Dgr} w_{1,n}^{J} - \boR e^{it\Dgr} w_{1,n}^J \\
e^{it\Dgr} w_{2,n}^{J} - \boR e^{it\Dgr} w_{2,n}^J \\
e^{it\Dgr} w_{3,n}^{J} + \boR e^{it\Dgr} w_{3,n}^J 
\end{bmatrix}
\bcp.
\end{align*}
which, combined with \eqref{Small:w1}, leads to (iii). \bigskip\par
The remaining part is to show (iv). 
We first consider the case when the norm is $L^2(\boG)$ or $\dot{H}_\gam^1(\boG)$. 
For this sake, it suffices to show
\begin{equation}\label{Orth:TT}
\inp{\boT_{t_n^j, y_n^j} (\psi_{1}^j,\psi_{2}^j,\psi_{3}^j)}{
\boT_{t_n^k, y_n^k} (\psi_{1}^k,\psi_{2}^k,\psi_{3}^k)}_X
\xrightarrow{n\to\infty} 0\qquad \text{for}\quad  j\neq k,
\end{equation}
\begin{equation}\label{Orth:Tw}
\inp{\boT_{t_n^j, y_n^j} (\psi_{1}^j,\psi_{2}^j,\psi_{3}^j)}{
\bow^J_n}_X
\xrightarrow{n\to\infty} 0\qquad \text{for}\quad  j\in\N,\ J\in\N.
\end{equation}
for $X=L^2(\boG)$, $\dot{H}^1_\gam (\boG)$. 
%
%
We start with the case $X=L^2(\boG)$. 
By \eqref{shift}, the left hand side of \eqref{Orth:TT} is a linear combination of sequences of the form
\begin{equation}\label{Piece:L2}
\inp{e^{-it_n^j \Dgr} \tau_{\pm y_n^j}\vph_j \bcp}{e^{-it_n^k\Dgr} \tau_{\pm y_n^k}\vph_k \bcp}_{L^2(\R)} \quad (\vph_j,\vph_k\in H^1(\R)).
\end{equation}
Since $|y_n^j \pm y_n^k| \ge |y_n^j-y_n^k|$, we have 
$
|t_n^j -t_n^k| + |y_n^j \pm y_n^k| \to \infty.
$ 
Thus it follows from Lemma \ref{Lem:Hgam} that each of \eqref{Piece:L2} goes to $0$ as $n\to\infty$, which implies \eqref{Orth:TT}. 
For \eqref{Orth:Tw}, we have
\begin{align*}
&\inp{\boT_{t_n^j, y_n^j} (\psi_{1}^j,\psi_{2}^j,\psi_{3}^j)}{
\bow^J_n}_{L^2(\R)} 
=
\inp{\boT_{0, y_n^j} (\psi_{1}^j,\psi_{2}^j,\psi_{3}^j)}{
e^{it_n^j\Dgg} \bow^J_n}_{L^2(\R)}\\
&=
\inp{\boT_{0, y_n^j} (\psi_{1}^j,\psi_{2}^j,\psi_{3}^j)}{
\mathcal{Q}^{-1} e^{it_n^j \Dgr} 
\begin{bmatrix}
w_{1,n}^{J} - \boR w_{1,n}^J  \\
w_{2,n}^{J} - \boR w_{2,n}^J \\
w_{3,n}^{J} + \boR w_{3,n}^J 
\end{bmatrix}
}_{L^2(\R)}
\end{align*}
which can be written as linear combination of the sequence of the form
\begin{align*}
\inp{\tau_{y_n^j} \vph_{j,1} \bcp}{e^{it_n^j\Dgr} (w_{k,n}^J \mp \boR w_{k,n}^J)}_{L^2(\R)},
&&
\inp{\tau_{-y_n^j} \vph_{j,2} \bcp}{e^{it_n^j\Dgr} (w_{k,n}^J \mp \boR w_{k,n}^J)}_{L^2(\R)}
\end{align*}
with $\vph_{j,1}, \vph_{j,2}\in H^1(\mathbb{R})$. 
The former component can be written as
\begin{align*}
&\inp{\tau_{y_n^j} \vph_{j,1} \bcp}{e^{it_n^j\Dgr} (w_{k,n}^J \mp \boR w_{k,n}^J) \bcp}_{L^2(\R)} \\
&=
%
\inp{\vph_{j,1}}{\tau_{-y_n^j} e^{it_n^j \Dgr} (w_{k,n}^J \mp \boR w_{k,n}^J)}_{L^2(\R)} 
-
\inp{\vph_{j,1} \mathbbm{1}_{<-y_n^j}
}{\tau_{-y_n^j} e^{it_n^j \Dgr} (w_{k,n}^J \mp \boR w_{k,n}^J)}_{L^2(\R)}.
\end{align*}
Then the first term converges to $0$ due to \eqref{Small:w2}, 
while we have
\begin{align*}
&\left|\inp{\vph_{j,1} \mathbbm{1}_{<-y_n^j}
}{\tau_{-y_n^j} e^{it_n^j \Dgr} (w_{k,n}^J \mp \boR w_{k,n}^J)}_{L^2(\R)} \right|\\
&\hspace{70pt}\le 
\left( \sup_{n\in\N} \nor{w_{k,n}^J \mp \boR w_{k,n}^J}{L^2(\R)} \right) 
\nor{\vph_{j,1}}{L^2(-\infty,-y_n^j)}
\end{align*}
which goes to $0$, where the boundedness of $\nor{w_{k,n}^J}{L^2(\R)}$ is a consequence of \eqref{Small:w2}. 
The latter component also vanishes as $n\to\infty$ since 
$\tau_{-y_n^j} \vph_{j,2} \bcp \to 0$ in $H^1(\R)$. 
Hence \eqref{Orth:Tw} follows.\medskip\par
Next we consider the case $X=\dot{H}^1_\gam (\boG)$. 
The left hand side of \eqref{Orth:TT} can be written as 
\begin{equation}\label{Eq:Orth:TTH1}
\begin{aligned}
&\inp{\boT_{t_n^j-t_n^k, y_n^j} (\psi_{1}^j,\psi_{2}^j,\psi_{3}^j)}{
\boT_{0, y_n^k} (\psi_{1}^k,\psi_{2}^k,\psi_{3}^k)}_{\dot{H}^1_\gam} 
\\
&=
\inp{ 
\rd_x \boT_{t_n^j-t_n^k, y_n^j} (\psi_{1}^j,\psi_{2}^j,\psi_{3}^j)}{
\rd_x\boT_{0, y_n^k} (\psi_{1}^k,\psi_{2}^k,\psi_{3}^k)}_{L^2_x} 
\\
& \hspace{70pt}
+ 3\gam \left|
\boT_{t_n^j-t_n^k, y_n^j} (\psi_{1}^j,\psi_{2}^j,\psi_{3}^j) (0)
\right|
\left|
\boT_{0, y_n^k} (\psi_{1}^k,\psi_{2}^k,\psi_{3}^k) (0)
\right|
.
\end{aligned}
\end{equation}
That the first term of \eqref{Eq:Orth:TTH1} 
converges to $0$ follows similarly to the case $X=L^2(\boG)$. 
On the other hand, the second term can be written as linear combination of the sequence of the form
\begin{equation}\label{Piece:H1}
3\gam \left| e^{i(t_n^k-t_n^j) \Dgr} \tau_{\pm y_n^j} \vph_j (0)\right| 
\left| \tau_{\pm y_n^k} \vph_k (0)\right|\qquad
(\vph_j,\vph_k\in H^1(\R)).
\end{equation}
Noting that $e^{i(t_n^k-t_n^j)\Dgr} \tau_{\pm y_n^j} \vph_j \rightharpoonup 0$ in $H^1(\R)$, and that $\vph\mapsto \vph(0)$ is a bounded linear functional on $H^1(\R)$, we conclude that \eqref{Piece:H1} converges to $0$ as $n\to\infty$. 
Hence \eqref{Orth:TT} follows. 
For \eqref{Orth:Tw}, we can write
\begin{align*}
&\inp{\boT_{0,y_n^j} (\psi_1^j,\psi_2^j,\psi_3^j)}{e^{it_n^j\Dgg}\bow_n^J}_{\dot{H}^1_\gam}\\
&=
\inp{\rd_x\boT_{0,y_n^j} (\psi_1^j,\psi_2^j,\psi_3^j)}{\rd_x e^{it_n^j\Dgg}\bow_n^J}_{L^2}
+ 3\gam \left|\boT_{0,y_n^j} (\psi_1^j,\psi_2^j,\psi_3^j)(0)\right|
\left|e^{it_n^j\Dgg}\bow_n^J (0) \right|.
\end{align*}
Then the first term goes to $0$ as $n\to\infty$ by the similar argument. 
By \eqref{Small:w2}, 
the second term also converges to $0$. Therefore \eqref{Orth:Tw} follows and the case $X=\dot{H}^1_\gam (\boG)$ is completed. \medskip\par
Finally, we prove (iv) for $L^q(\boG)$-norm with $(2,\infty)$. 
It is standard that there exists a constant $C$ such that 
\begin{equation}\label{Est:InnLq}
\left| \left| \sum_{j=1}^l z_j\right|^q - \sum_{j=1}^l |z_j|^q \right| \le 
C \sum_{j\neq k} |z_j|^{q-1} |z_k|
\end{equation}
holds for any integer $l\ge 2$ and $z_1,\cdots z_l\in \C$. 
Hence by \eqref{shift}, it suffices to show that for all $\vph_1,\vph_2\in H^1(\R)$, 
\begin{equation}\label{Orth:TTq}
\int_0^\infty \left| e^{-it_n^j\Dgr} \tau_{\pm y_n^j}\vph_1 \right|^{q-1}
\left| e^{-it_n^k \Dgr} \tau_{\pm y_n^k}\vph_2 \right| dx 
\xrightarrow{n\to\infty} 0\qquad \text{for}\quad  j\neq k,
\end{equation}
\begin{equation}\label{Orth:Twq}
\begin{aligned}
\int_0^\infty \left| e^{-it_n^j\Dgr} \tau_{\pm y_n^j}\vph_1 \right|^{q-1}
\left| w_{k,n}^J \mp \boR w_{k,n}^J \right| 
&+
 \left| e^{-it_n^j\Dgr} \tau_{\pm y_n^j}\vph_1 \right|
\left| w_{k,n}^J \mp \boR w_{k,n}^J \right|^{q-1} dx \\
&\xrightarrow{n\to\infty} 0\qquad \text{for}\quad  j\in\N,\ J\in\N.
\end{aligned}
\end{equation}
For this sake we may suppose $\vph_1,\vph_2\in C_0^\infty(\R)$ by approximation. 
For \eqref{Orth:TTq}, it suffices to consider either of the cases when 
(I) $t_n^j$ or $t_n^k\to\pm\infty$ or (II) $t_n^j=t_n^k=0$ for all $n$. 
If (I), then \eqref{Orth:TTq} easily follows from the dispersive estimate. 
If (II), then we necessarily have $|y_n^j\pm y_n^k|\to\infty$, and hence \eqref{Orth:TTq} holds. 
For \eqref{Orth:Twq}, 
we only consider the $\tau_{+y_n^j}$ case of the first term, while the others can be shown similarly. 
We may again suppose $\vph_1\in C_0^\infty(\R)$. 
If $t_n^j\to\pm \infty$, then \eqref{Orth:Twq} is an immediate consequence of the dispersive estimate and \eqref{Small:w2}. 
Thus we may assume $t_n^j= 0$ for all $n$. Then we can write
$$
\begin{aligned}
&\int_0^\infty \left|\tau_{y_n^j}\vph_1 \right|^{q-1}
\left| w_{k,n}^J \mp \boR w_{k,n}^J \right| dx \\
&\hspace{60pt}=
\left(\int_{-\infty}^\infty - \int_{-\infty}^{-y_n^j}\right)
\left|\vph_1 \right|^{q-1}
\left| \tau_{-y_n^j} (w_{k,n}^J \mp \boR w_{k,n}^J) \right| dx.
\end{aligned}
$$
The second integral clearly converges to $0$ since $-y_n^j\to-\infty$. 
On the other hand, by \eqref{Small:w2} and the compactness of $\supp \vph_1$, the Rellich-Kondrachov theorem implies that the first term also goes to $0$. 
Hence the proof is completed.


\section{Propositions for nonlinear profile decomposition}\label{Sec:Non}

In this section we establish propositions which are necessary for the nonlinear profile decompositions for \eqref{NLS}. 
We first observe the behavior of norms for spatially shifting profiles.

\begin{prop}
\label{prop5.1}
Let $\psi_k\in H^1(\R)$ ($k=1,2,3$), and let 
$\{t_n\}_{n=1}^\infty$, $\{y_n\}_{n=1}^\infty$ be sequences with 
$y_n\to\infty$. Then we have
\begin{equation}\label{Orth:TL2}
\lim_{n\to \infty} \nor{\boT_{t_n,y_n} (\psi_1,\psi_2,\psi_3)}{L^2(\boG)}^2
=
\nor{\psi_1}{L^2(\R)}^2 + \nor{\psi_2}{L^2(\R)}^2 + \nor{\psi_3}{L^2(\R)}^2,
\end{equation}
\begin{equation}\label{Orth:TH1}
\lim_{n\to \infty} \nor{\boT_{t_n,y_n} (\psi_1,\psi_2,\psi_3)}{\dot{H}_\gam^1(\boG)}^2
=
\nor{\psi_1}{\dot{H}^1(\R)}^2 + \nor{\psi_2}{\dot{H}^1(\R)}^2 + \nor{\psi_3}{\dot{H}^1(\R)}^2.
\end{equation}
Moreover, for $q\in (2,\infty)$, we have
\begin{equation}\label{Orth:TLq}
\lim_{n\to \infty} \nor{\boT_{0,y_n} (\psi_1,\psi_2,\psi_3)}{L^q(\boG)}^q
=
\nor{\psi_1}{L^q(\R)}^q + \nor{\psi_2}{L^q(\R)}^q + \nor{\psi_3}{L^q(\R)}^q.
\end{equation}
\end{prop}

\begin{proof}
By \eqref{shift}, we can write
$$
\boT_{t_n,y_n}(\psi_1,\psi_2,\psi_3) = 
\begin{pmatrix}
e^{-it_n\Dgr} \tau_{y_n} \psi_1 \bcp \\
e^{-it_n\Dgr} \tau_{y_n} \psi_2 \bcp \\
e^{-it_n\Dgr} \tau_{y_n} \psi_3 \bcp 
\end{pmatrix}
+ 
\bG_{t_n,y_n} (\psi_1,\psi_2,\psi_3).
$$
where
\begin{equation}\label{Def:G}
\bG_{t,y} (\vph_1,\vph_2,\vph_3) :=
\begin{pmatrix}
\frac 13 e^{-it\Dgr} \tau_{-y} \boR (-\psi_1 +2\psi_2 +2\psi_3) \bcp\\
\frac 13 e^{-it\Dgr} \tau_{-y} \boR (2\psi_1 -\psi_2 +2\psi_3) \bcp\\
\frac 13 e^{-it\Dgr} \tau_{-y} \boR (2\psi_1 +2\psi_2 -\psi_3) \bcp
\end{pmatrix}
,\quad (t\in\R, y\ge 0)
.
\end{equation}
Then we have
\begin{equation}\label{Small:G}
\nor{\bG_{0,y} (\psi_1,\psi_2,\psi_3)}{H^1 (\boG)\cap L^q(\boG)} 
 \xrightarrow{y\to\infty} 0.
\end{equation}
Since $e^{it\Dgg}$ is unitary on $L^2(\boG)$, we have
\begin{align*}
\lim_{n\to\infty} \nor{\boT_{t_n,y_n} (\psi_1,\psi_2,\psi_3)}{L^2(\boG)}^2 
&=
\lim_{n\to\infty} \nor{\boT_{0,y_n} (\psi_1,\psi_2,\psi_3)}{L^2(\boG)}^2 \\
&=
\lim_{n\to\infty}
\nor{
\begin{pmatrix}
\tau_{y_n} \psi_1 \bcp \\
\tau_{y_n} \psi_2 \bcp \\
\tau_{y_n} \psi_3 \bcp
\end{pmatrix}
}{L^2(\boG)}^2 \\
&= 
\nor{\psi_1}{L^2(\R)}^2 + \nor{\psi_2}{L^2(\R)}^2 + \nor{\psi_3}{L^2(\R)}^2,
\end{align*}
which leads to \eqref{Orth:TL2}. Next, we have
\begin{align*}
&\lim_{n\to\infty} \nor{\boT_{t_n,y_n} (\psi_1,\psi_2,\psi_3)}{\dot{H}^1_\gam(\boG)}^2 
=
\lim_{n\to\infty} \nor{\boT_{0,y_n} (\psi_1,\psi_2,\psi_3)}{\dot{H}^1_\gam(\boG)}^2\\
&=
\lim_{n\to\infty} (
\nor{\rd_x \boT_{0,y_n} (\psi_1,\psi_2,\psi_3)}{L^{2}(\boG)}^2 
+3\gam |\boT_{0,y_n} (\psi_1,\psi_2,\psi_3) (0)|^2).
\end{align*}
Then the same argument as above yields
$$
\lim_{n\to\infty} \nor{\rd_x \boT_{0,y_n} (\psi_1,\psi_2,\psi_3)}{L^{2}(\boG)}^2 
=
\nor{\psi_1}{\dot{H}^1(\R)}^2 + \nor{\psi_2}{\dot{H}^1(\R)}^2 + \nor{\psi_3}{\dot{H}^1(\R)}^2.
$$
On the other hand, we have
$$
|\boT_{0,y_n} (\psi_1,\psi_2,\psi_3) (0)| = \frac 23  
\left|
\psi_1(-y_n) + \psi_2(-y_n) + \psi_3 (-y_n)
\right|
\xrightarrow{n\to\infty} 0.
$$
Hence \eqref{Orth:TH1} follows. 
Finally, \eqref{Small:G} and \eqref{Est:InnLq} give us 
\begin{align*}
\lim_{n\to\infty} \nor{\boT_{0,y_n} (\psi_1,\psi_2,\psi_3)}{L^q(\boG)}^q 
&=
\lim_{n\to\infty}
\nor{
\begin{pmatrix}
\tau_{y_n} \psi_1 \bcp \\
\tau_{y_n} \psi_2 \bcp \\
\tau_{y_n} \psi_3 \bcp
\end{pmatrix}
}{L^q(\boG)}^q \\
&= 
\nor{\psi_1}{L^q(\R)}^q + \nor{\psi_2}{L^q(\R)}^q + \nor{\psi_3}{L^q(\R)}^q
\end{align*}
for $q\in(2,\infty)$, which concludes \eqref{Orth:TLq}. 
\end{proof}


\begin{lemma}
\label{lem4.3}
Let $f_{k}^{j} \in L_{t}^{a}(\mathbb{R};L_{x}^{r}(\mathbb{R}))$ for $1\leq j \leq J$ and $k=1,2,3$. Assume that $|x_{n}^{j}-x_{n}^{j'}|+|t_{n}^{j}-t_{n}^{j'}| \to \infty$ as $n \to \infty$ for $j \neq j'$. Then we have 
\begin{align*}
	\limsup_{n \to \infty} \left( \|\sum_{j=1}^{J}\mathcal{T}_{0,x_{n}^{j}}(f_{1}^{j},f_{2}^{j},f_{3}^{j})(t-t_{n}^{j})\|_{L_{t}^{a}L_{x}^{r}}\right)^{p}
	\lesssim  \sum_{j=1}^{J} \sum_{k=1}^{3} \|f_{k}^{j}\|_{L_{t}^{a}L_{x}^{r}}^{p}
\end{align*}
\end{lemma}

\begin{proof}
We have 
\begin{align*}
	&\left( \|\sum_{j=1}^{J}\mathcal{T}_{0,x_{n}^{j}}(f_{1}^{j},f_{2}^{j},f_{3}^{j})(t-t_{n}^{j})\|_{L_{t}^{a}L_{x}^{r}}\right)^{p}\\
	&\leq  \|  |\sum_{j=1}^{J}\mathcal{T}_{0,x_{n}^{j}}(f_{1}^{j},f_{2}^{j},f_{3}^{j})(t-t_{n}^{j}) |^{p} \|_{L_{t}^{b'}L_{x}^{r'}}\\
	&\leq \| |\sum_{j=1}^{J}\mathcal{T}_{0,x_{n}^{j}}(f_{1}^{j},f_{2}^{j},f_{3}^{j})(t-t_{n}^{j})|^{p}-\sum_{j=1}^{J}|\mathcal{T}_{0,x_{n}^{j}}(f_{1}^{j},f_{2}^{j},f_{3}^{j})(t-t_{n}^{j}) |^{p}\|_{L_{t}^{b'}L_{x}^{r'}} \\
	&+\| \sum_{j=1}^{J}|\mathcal{T}_{0,x_{n}^{j}}(f_{1}^{j},f_{2}^{j},f_{3}^{j})(t-t_{n}^{j}) |^{p}\|_{L_{t}^{b'}L_{x}^{r'}}.
\end{align*}
The second term is estimated as follows:
\begin{align*}
	\| \sum_{j=1}^{J}|\mathcal{T}_{0,x_{n}^{j}}(f_{1}^{j},f_{2}^{j},f_{3}^{j})(t-t_{n}^{j}) |^{p}\|_{L_{t}^{b'}L_{x}^{r'}} 
	&\leq \sum_{j=1}^{J} \|\mathcal{T}_{0,x_{n}^{j}}(f_{1}^{j},f_{2}^{j},f_{3}^{j})(t-t_{n}^{j}) \|_{L_{t}^{a}L_{x}^{r}}^{p} \\
	&\leq \sum_{j=1}^{J} \|\mathcal{T}_{0,x_{n}^{j}}(f_{1}^{j},f_{2}^{j},f_{3}^{j})\|_{L_{t}^{a}L_{x}^{r}}^{p}\\
	&\lesssim \sum_{j=1}^{J} \sum_{k=1}^{3} \|f_{k}^{j}\|_{L_{t}^{a}L_{x}^{r}}^{p}.
\end{align*}
The first term goes to $0$ as $n \to \infty$ by \eqref{Est:InnLq} and Lemma \ref{lem4.4} below. 
\end{proof}

\begin{lemma}
\label{lem4.4}
Let $\bm{u}^{1},\bm{u}^{2} \in CH_{c}^{1} \cap L_{t}^{p}L_{x}^{r}$. 
Let $\{t_{n}^{j}\}$ and $\{x_{n}^{j}\}$ ($j=1,2$) satisfy 
\begin{align*}
	|t_{n}^{1}-t_{n}^{2}|+|x_{n}^{1}-x_{n}^{2}| \to \infty
\end{align*}
as $n \to \infty$. Then it holds that
\begin{align*}
	\| |\mathcal{T}_{0,x_{n}^{1}}\bm{u}^{1}(t-t_{n}^{1})|^{p-1}|\mathcal{T}_{0,x_{n}^{2}}\bm{u}^{2}(t-t_{n}^{2})|\|_{L_{t}^{b'}L_{x}^{r'}} \to 0
\end{align*}
as $n \to \infty$. 
\end{lemma}

See Proposition A.1 in \cite{BanIgn14} for the proof. 

Next we construct a solution to \eqref{NLS} which approximates the linear solution with initial data being each profile of \eqref{Eq:LPD}. 

\begin{prop}\label{Prop:NPD1}
Let $\{t_n\}_{n\in\N} \subset\R$ with $t_n\to\pm \infty$. 
Let $\bvph\in H^1_c(\boG)$, 
and suppose that there exists a time-in-global solution $\bm{W}_{\pm}\in C_tH^1_c \cap L^a_tL^r_x$ to \eqref{NLS} satisfying 
$$
\nor{\bW_\pm (t,\cdot) - e^{it\Dgg}\bvph}{H^1(\boG)} \to 0 \qquad (t\to \mp \infty).
$$
 Then we have
$$
\bW_n = e^{it\Dgg} (e^{-it_n\Dgg} 
\bvph) - i \int_0^t e^{i(t-s)\Dgg}\mathcal{N}(\bW_n) ds + \bE_n,
$$
where
$$
\bW_n (t,x) := \bW_{\pm}(t-t_n,x),
$$
$$
\nor{\bE_n}{L^a_tL^r_x} \to 0\qquad (n\to\infty).
$$
\end{prop}

\begin{proof}
We only prove in the case $t_n\to\infty$, and omit the index $\pm$. 
Since $\bW_n$ satisfies
$$
\bW_n = e^{it\Dgg} \bW(-t_n) - i \int_0^t e^{i(t-s)\Dgg}\mathcal{N}(\bW_n)ds,
$$
we have
$$
\bE_n = e^{it\Dgg} (\bW(-t_n) - e^{-it_n\Dgg}\vph).
$$
Hence $\nor{\bE_n}{L^a_tL^r_x}\to 0$ by the Strichartz estimate. 
\end{proof}

As stated below, the profiles with $y_n \to \infty$ ($n \to \infty$) can be approximated by the solution to NLS without any linear potentials on $\mathbb{R}$. The equation is 
\begin{align}
\label{NLSL}
	i\partial_t u + \partial_{xx} u = N(u), \quad (t,x) \in \mathbb{R} \times \mathbb{R},
\end{align}
where $N(u)=-|u|^{p-1}u$.

\begin{prop}\label{Prop:NPD2}
Let $u_{j}\in C_tH^1\cap L^a_tL^r_x$ be a time-in-global mild solution of \eqref{NLSL} and with initial data $\psi_j$ ($j=1,2,3$): 
\begin{equation}\label{NLSL0:mild}
u_{j} = e^{it\rd_{xx}} \psi_j - i \int_0^t e^{i(t-s) \rd_{xx}} 
N(u_j) ds \qquad (j=1,2,3).
\end{equation}
Also let $\{y_n \}_{n\in\N}\subset (0,\infty)$ be a sequence with $y_n\to\infty$. 
Then,
$$
\bU_{n} = e^{it\Dgg} \boT_{0,y_n} (\psi_1,\psi_2,\psi_3) - i \int_0^t e^{i(t-s)\Dgg } 
\boN(\bU_n) ds +\bE_n,
$$
where
$$
\bU_n := 
\boT_{0,y_n}(u_1,u_2,u_3)
$$
and
$$
\nor{\bE_n}{L^a_tL^r_x(\boG)} \to 0\qquad (n\to\infty).
$$
\end{prop}

\begin{proof}
Using \eqref{Def:G}, we can write
$$
\bU_n = \begin{pmatrix}
\tau_{y_n} u_1 \bcp \\
\tau_{y_n} u_2 \bcp \\
\tau_{y_n} u_3 \bcp 
\end{pmatrix}
+ 
\bG_{0,y_n} (u_1,u_2,u_3).
$$
Then by \eqref{NLSL0:mild}, we can write
$$
\bE_n = \bA_{1,n} + \bA_{2,n} - \bA_{3,n}
$$
where
\begin{align*}
\bA_{1,n} := 
\begin{pmatrix}
e^{it\rd_{xx} } \tau_{y_n} \psi_1 \bcp\\
e^{it\rd_{xx} } \tau_{y_n} \psi_2 \bcp\\
e^{it\rd_{xx} } \tau_{y_n} \psi_3 \bcp
\end{pmatrix}
-
e^{it\Dgg} \boT_{0,y_n}(\psi_1,\psi_2,\psi_3) ,
&&
\bA_{2,n}:=\bG_{0,y_n} (u_1,u_2,u_3),
\end{align*}
$$
\bA_{3,n}:=
\begin{pmatrix}
i \int_0^t e^{i(t-s)\rd_{xx}} \tau_{y_n} N(u_1) ds \bcp\\
i \int_0^t e^{i(t-s)\rd_{xx}} \tau_{y_n} N(u_2)  ds \bcp\\
i \int_0^t e^{i(t-s)\rd_{xx}} \tau_{y_n} N(u_3)  ds \bcp\\
\end{pmatrix}
-
i \int_0^t e^{i(t-s)\Dgg } \boN (\bU_n) ds.
$$
Thus it suffices to show each of these converges to $0$ in $L^a_tL^r_x$. 
First, $\bA_{1,n}$ can be written as
$$
\bA_{1,n} = 
\begin{pmatrix}
[e^{it\rd_{xx} } - e^{it\Dgr } ] \tau_{y_n} \psi_1 \bcp\\
[e^{it\rd_{xx} } - e^{it\Dgr } ] \tau_{y_n} \psi_2 \bcp\\
[e^{it\rd_{xx} } - e^{it\Dgr } ] \tau_{y_n} \psi_3 \bcp
\end{pmatrix}
- \bG_{-t,y_n} (\psi_1,\psi_2, \psi_3),
$$
which goes to $0$ by Lemma \ref{Lem:Hgam}. 
Since $u_k\in L^a_tL^r_x$, 
we also have
$$
\nor{\bA_{2,n}}{L^a_tL^r_x} 
\le C \sum_{j=1}^3 \nor{\tau_{-y_n} \boR u_j \bcp (x)}{L^a_tL^r_x}
= C\sum_{j=1}^3 \nor{\boR u_j \mathbbm{1}_{>y_n}(x)}{L^a_tL^r_x}
\xrightarrow{n\to \infty} 0.
$$
For $\bA_{3,n}$, we can further decompose
$$
\bA_{3,n} = 
\bA_{31,n} +\bA_{32,n} +\bA_{33,n} +\bA_{34,n}
$$
where
$$
\bA_{31,n} 
:=
\begin{pmatrix}
i \int_0^t [e^{i(t-s)\rd_{xx}} -e^{i(t-s)\Dgr}] \tau_{y_n} N(u_1)  ds \bcp\\
i \int_0^t [e^{i(t-s)\rd_{xx}} -e^{i(t-s)\Dgr}] \tau_{y_n} N(u_2)  ds \bcp\\
i \int_0^t [e^{i(t-s)\rd_{xx}} -e^{i(t-s)\Dgr}] \tau_{y_n} N(u_3)  ds \bcp\\
\end{pmatrix}
,
$$
$$
\bA_{32,n} 
:=
- i\int_0^t \bG_{s-t,y_n} (N(u_1), N(u_2) , N(u_3)) ds,
$$
$$
\bA_{33,n}
:=
i\int_0^t e^{i(t-s) \Dgg} \bG_{0,y_n} (N(u_1), N(u_2) , N(u_3))ds,
$$
$$
\bA_{34,n}:= -i\int_0^t e^{i(t-s) \Dgg} 
\begin{pmatrix}
N(U_1) - \tau_{y_n} N(u_1) \bcp \\
N(U_2) - \tau_{y_n} N(u_2) \bcp \\
N(U_3) - \tau_{y_n} N(u_3) \bcp 
\end{pmatrix}
ds.
$$
First, 
$\bA_{31,n}\to 0$ in $L^a_tL^r_x$ follows from Lemma \ref{Lem:Hgam}. 
For $\bA_{32,n}$, the integrand can be written as
\begin{align*}
&
\begin{pmatrix}
\frac 13 [e^{i(t-s)\Dgr} -e^{i(t-s)\rd_{xx}}] \tau_{-y_n} \boR (-N(u_1) +2N(u_2) +2N(u_3)) \bcp\\
\frac 13 [e^{i(t-s)\Dgr} -e^{i(t-s)\rd_{xx}}] \tau_{-y_n} \boR (2N(u_1) -N(u_2) +2N(u_3)) \bcp\\
\frac 13 [e^{i(t-s)\Dgr} -e^{i(t-s)\rd_{xx}}] \tau_{-y_n} \boR (2N(u_1) +2N(u_2) -N(u_3)) \bcp
\end{pmatrix}
\\
&\hspace{20pt} +
\begin{pmatrix}
[\tau_{-y_n} \frac 13 e^{i(t-s)\rd_{xx}}  \boR (-N(u_1) +2N(u_2) +2N(u_3))] \bcp\\
[\tau_{-y_n} \frac 13 e^{i(t-s)\rd_{xx}} \boR (2N(u_1) -N(u_2) +2N(u_3))] \bcp\\
[\tau_{-y_n} \frac 13 e^{i(t-s)\rd_{xx}}  \boR (2N(u_1) +2N(u_2) -N(u_3))] \bcp
\end{pmatrix}
.
\end{align*}
Hence Lemma \ref{Lem:Hgam} and the Strichartz estimate for $e^{it\Delta^{\gamma}_{\mathbb{R}}}$ imply 
$\bA_{32,n}\to 0$ in $L^a_tL^r_x$. 
That $\bA_{33,n}\to 0$ follows from the Strichartz estimate for $e^{it\Dgg}$. 
Finally, $\bA_{34,n}$ is estimated as
\begin{align*}
\nor{A_{34,n}}{L^a_tL^r_x}
&\le
C \sum_{\nu=1}^3 \nor{N(U_{\nu,n}) -N(\tau_{y_n}u_\nu)\bcp}{L^{a/p}_tL^{r/p}_x} \\
&\le 
C \sum_{\nu=1}^3 (\nor{U_{\nu,n}}{L^a_tL^r_x} +\nor{\tau_{y_n}u_\nu}{L^p_tL^r_x})^{p-1} 
\nor{U_{\nu,n}- \tau_{y_n} u_\nu \bcp}{L^a_tL^r_x} \\
&\le
C \left( \sum_{\nu=1}^3 \nor{u_\nu}{L^a_tL^r_x} \right)^{p-1} 
\nor{\bG_{0,y_n}}{L^a_tL^r_x}
\to 0\qquad (n\to\infty).
\end{align*}
Hence the proof is completed.
\end{proof}

\begin{prop}\label{Prop:NPD3}
Let $\{ t_n \}_{n=1}^\infty$, $\{y_n \}_{n=1}^\infty$ be sequences with $t_n\to \pm\infty$ and $y_n\to\infty$. 
Let $\vph_k\in H^1(\R)$ $(k=1,2,3)$ be functions, and suppose that 
there exist global solutions $v_{\pm, k} \in C_tH^1\cap L^a_tL^r_x$ to \eqref{NLSL} with 
\begin{equation}\label{v:scatter}
\nor{v_{\pm, k} - e^{it\rd_{xx}} \vph_k}{H^1(\R)} \to 0\qquad  (t\to \mp \infty, \ k=1,2,3).
\end{equation}
Then we have
$$
\bV_n = e^{it\Dgg} \boT_{t_n,y_n} (\vph_1,\vph_2,\vph_3) - i \int_0^t e^{i(t-s)\Dgg} 
\boN (\bV_n) ds +\bE_n
$$
where
$$
\bV_n (t,x) := 
\boT_{0,y_n} (v_{\pm,1},v_{\pm,2},v_{\pm,3}) (t-t_n),
$$
and
$$
\nor{\bE_n}{L^a_tL^r_x}\to 0 \qquad (n\to\infty).
$$
\end{prop}

\begin{proof}
We omit the suffix $\pm$ for short. 
Using the notation \eqref{Def:G}, we can write
$$
\bE_n = \bB_{1,n} + \bB_{2,n} - \bB_{3,n}
$$
where
\begin{align*}
\bB_{1,n} &:= 
\begin{pmatrix}
e^{it\rd_{xx}} \tau_{y_n} v_{1} (-t_n) \\
e^{it\rd_{xx}} \tau_{y_n} v_{2} (-t_n) \\
e^{it\rd_{xx}} \tau_{y_n} v_{3} (-t_n) 
\end{pmatrix}
\bcp
-
e^{it\Dgg} \boT_{t_n,y_n} (\vph_1,\vph_2,\vph_3),
\\
\bB_{2,n} &:= \bG_{0,y_n}(v_1,v_2,v_3) (t-t_n),
\end{align*}
$$
\bB_{3,n}:=
\begin{pmatrix}
i\int_{0}^t e^{i(t-s)\rd_{xx}}  N(\tau_{y_n} v_1|_{s-t_n}) ds \\
i\int_{0}^t e^{i(t-s)\rd_{xx}}  N(\tau_{y_n} v_2|_{s-t_n}) ds \\
i\int_{0}^t e^{i(t-s)\rd_{xx}}  N(\tau_{y_n} v_3|_{s-t_n}) ds 
\end{pmatrix}
\bcp
-
i \int_0^t e^{i(t-s)\Dgg} 
\boN(\bV_n) ds.
$$
It suffices to show that $\bB_{j,n}\to 0$ in $L^a_tL^r_x$ for $j=1,2,3$. 
For $\bB_{1,n}$, we further decompose
\begin{align*}
\bB_{1,n}
=&
\begin{pmatrix}
e^{it\rd_{xx}} \tau_{y_n} [v_{1} (-t_n) - e^{-it_n\rd_{xx}} \vph_1 ] \\
e^{it\rd_{xx}} \tau_{y_n} [v_{2} (-t_n) - e^{-it_n\rd_{xx}} \vph_2 ] \\
e^{it\rd_{xx}} \tau_{y_n} [v_{3} (-t_n) - e^{-it_n\rd_{xx}} \vph_3 ]
\end{pmatrix}
\bcp
\\
&+
\begin{pmatrix}
[e^{i(t-t_n)\rd_{xx}} - e^{i(t-t_n)\Dgr}] \tau_{y_n} \vph_1 \\
[e^{i(t-t_n)\rd_{xx}} - e^{i(t-t_n)\Dgr}] \tau_{y_n} \vph_2 \\
[e^{i(t-t_n)\rd_{xx}} - e^{i(t-t_n)\Dgr}] \tau_{y_n} \vph_3 
\end{pmatrix}
\bcp
- 
\bG_{t_n-t,y_n} (\vph_1,\vph_2,\vph_3).
\end{align*}
The convergence of the first and second term follows from \eqref{v:scatter}, 
Lemma \ref{Lem:Hgam}, respectively. 
The last term can be shown to go to $0$ by a similar argument to that of $\bA_{32,n}$ in the proof of Proposition \ref{Prop:NPD2}. \par
Since $v_k\in L^a_tL^r_x$, that $\bB_{2,n}\to 0$ follows in the same way as $\bA_{2.n}$. For $\bB_{3,n}$, we write
$$
\bB_{3,n} 
=
\bB_{31,n} +\bB_{32,n} +\bB_{33,n} +\bB_{34,n}
$$
where
$$
\bB_{31,n} :=
\begin{pmatrix}
i \int_0^t [e^{i(t-s)\rd_{xx}} -e^{i(t-s)\Dgr}] N(\tau_{y_n} v_1|_{s-t_n} )  ds \\
i \int_0^t [e^{i(t-s)\rd_{xx}} -e^{i(t-s)\Dgr}] N(\tau_{y_n} v_2|_{s-t_n} ) ds \\
i \int_0^t [e^{i(t-s)\rd_{xx}} -e^{i(t-s)\Dgr}] N(\tau_{y_n} v_3|_{s-t_n} ) ds \\
\end{pmatrix}
\bcp
,
$$
$$
\bB_{32,n} :=
i\int_0^t e^{i(t-s) \Dgg} \bG_{0,y_n} (N(v_1|_{s-t_n}), N(v_2|_{s-t_n}) , N(v_3|_{s-t_n}))ds,
$$
$$
\bB_{33,n} :=
- i\int_0^t \bG_{s-t,y_n} (N(v_1|_{s-t_n}), N(v_2|_{s-t_n}) , N(v_3|_{s-t_n})) ds,
$$
$$
\bB_{34,n} :=
-i\int_0^t e^{i(t-s) \Dgg} 
\begin{pmatrix}
N(V_{1,n}|_s) -N(\tau_{y_n} v_1|_{s-t_n}) \bcp\\
N(V_{2,n}|_s) -N(\tau_{y_n} v_2|_{s-t_n}) \bcp\\
N(V_{3,n}|_s) -N(\tau_{y_n} v_3|_{s-t_n}) \bcp
\end{pmatrix}
ds.
$$
The convergence of $\bB_{31,n}$ follows from Lemma \ref{Lem:Hgam}. 
Since we have
$$
\bB_{32,n}
=
i\int_{-t_n}^{t-t_n} e^{i(t-t_n -s) \Dgg} \bG_{0,y_n} (N(v_1|_{s}), N(v_2|_{s}) , N(v_3|_{s}))ds
,
$$
to prove $\bB_{32,n}\to 0$ it suffices to show
$$
i\int_{-t_n}^{t} e^{i(t -s) \Dgg} \bG_{0,y_n} (N(v_1|_{s}), N(v_2|_{s}) , N(v_3|_{s}))ds 
\xrightarrow{n\to\infty} 0\qquad \text{in } L^a_tL^r_x.
$$
When $t_{n}\to \infty$, we split the integral into
$$
i\left[ \int_{-\infty}^{t} - \int_{-\infty}^{-t_n} \right] e^{i(t -s) \Dgg} \bG_{0,y_n} (N(v_1|_{s}), N(v_2|_{s}) , N(v_3|_{s}))ds.
$$
Then the first term goes to $0$ by Srtichartz, while the second term also converges to $0$ since $-t_n\to-\infty$. The case $t_{n}\to -\infty$ follows in the same way. For $\bB_{33,n}$, we can write
\begin{align*}
&\int_0^t \bG_{s-t,y_n} (N(v_1|_{s-t_n}),N(v_2|_{s-t_n}),N(v_3|_{s-t_n}))\\
&\hspace{20pt} =
\begin{pmatrix}
\frac 13 \int_0^t [e^{i(t-s)\Dgr} -e^{i(t-s)\rd_{xx}}] \tau_{-y_n} \boR (-N(v_1) +2N(v_2) +2N(v_3))|_{s-t_n} \\
\frac 13 \int_0^t [e^{i(t-s)\Dgr} -e^{i(t-s)\rd_{xx}}] \tau_{-y_n} \boR (2N(v_1) -N(v_2) +2N(v_3))|_{s-t_n} \\
\frac 13 \int_0^t [e^{i(t-s)\Dgr} -e^{i(t-s)\rd_{xx}}] \tau_{-y_n} \boR (2N(v_1) +2N(v_2) -N(v_3))|_{s-t_n} 
\end{pmatrix}
\bcp
\\
&\hspace{40pt} +
\begin{pmatrix}
\tau_{-y_n} \frac 13 \int_0^t e^{i(t-s)\rd_{xx}}  \boR (-N(v_1) +2N(v_2) +2N(v_3))|_{s-t_n} \\
\tau_{-y_n} \frac 13 \int_0^t e^{i(t-s)\rd_{xx}} \boR (2N(v_1) -N(v_2) +2N(v_3))|_{s-t_n} \\
\tau_{-y_n} \frac 13 \int_0^t e^{i(t-s)\rd_{xx}}  \boR (2N(v_1) +2N(v_2) -N(v_3))|_{s-t_n} 
\end{pmatrix}
\bcp
.
\end{align*}
Then the convergence of the first terms follows from Lemma \ref{Lem:Hgam}, 
while that for the second term is shown in the same way as $\bB_{32,n}$. 
That $\bB_{34,n}\to 0$ follows similarly to $\bA_{34,n}$ in the proof of Prop 
\ref{Prop:NPD2}. Hence we obtain the conclusion.
\end{proof}

%
%

%
%
%
%


\section{Proof of Scattering}\label{Sec:Sca}
\subsection{Construction of a critical element}

We define the critical action $\mathcal{ME}^{c} \in \mathbb{R}$ by
\begin{align*}
	\mathcal{ME}^{c}
	:=\sup \{ L \in \mathbb{R} : 
	M(\bm{u})^{\frac{1-s_{c}}{s_{c}}}E_{\gamma}(\bm{u}) < L \text{ and } \bm{u} \in PW_{\gamma}^{+}
	\\
	\text{ implies } \bm{u} \in L^a(\mathbb{R}; L^r(\mathcal{G})) \}.
\end{align*}
The small data scattering result teaches us $\mathcal{ME}^{c}>0$.

\begin{lemma}[Existence of a critical element]
We assume that $\mathcal{ME}^{c} < M^{\mathrm{line}}(Q)^{(1-s_c)/s_c}E_{0}^{\mathrm{line}}(Q)$. Then, there exists a global solution $\bm{u}^c$  to \eqref{NLS} such that $M(\bm{u}^{c})^{(1-s_c)/s_c}E_{\gamma}(\bm{u}^{c})=\mathcal{ME}^{c}$ and $\|\bm{u}^c\|_{L^a(\mathbb{R}; L^r)}=\infty$.
\end{lemma}

The solution in the lemma is called a critical element. If the critical element satisfies $\|\bm{u}^c\|_{L^a((0,\infty): L^r)}=\infty$ (resp. $\|\bm{u}^c\|_{L^a((-\infty,0): L^r)}=\infty$), we call it forward (resp. backward) critical element.

\begin{proof}
Since we suppose that $\mathcal{ME}^{c} < M^{\mathrm{line}}(Q)^{(1-s_c)/s_c}E_{0}^{\mathrm{line}}(Q)$, we can take a sequence $\{\bm{\varphi}_{n}\} \subset PW_{\gamma}^{+}$ such that $M(\bm{\varphi}_{n})^{(1-s_c)/s_c}E_{\gamma}(\bm{\varphi}_{n}) \to \mathcal{ME}^{c}$ as $n \to \infty$ and $\|\bm{u}_{n}\|_{L^a((0,\infty); L^r)}=\infty$ for all $n \in \mathbb{N}$, where $\bm{u}_{n}$ is the global solution to \eqref{NLS} with the initial data $\bm{\varphi}_{n}$. 
Since $\{\bm{\varphi}_{n}\}$ is a bounded sequence in $H_{c}^{1}(\mathcal{G})$, applying the linear profile decomposition, we get the decomposition
\begin{align*}
	\bm{\varphi}_{n} = \sum_{j=1}^{J} \mathcal{T}_{t_{n}^{j},y_{n}^{j}} (\psi_{1,j}, \psi_{2,j}, \psi_{3,j}) + \bm{w}_{n}^{J}
\end{align*}
satisfying the properties in Proposition \ref{Prop:LPD}. We set $\bm{\psi}_{n}^{j}:=\mathcal{T}_{t_{n}^{j},y_{n}^{j}} (\psi_{1,j}, \psi_{2,j}, \psi_{3,j})$ for simplicity. 
By the orthogonality properties, for all $1\leq j' \leq J$, we have
\begin{align}
	\notag
	M( \bm{\varphi}_{n}) &= \sum_{j=1}^{J} M(\bm{\psi}_{n}^{j}) + M(\bm{w}_{n}^{J}) +o_{n}(1)
	\\ \label{eq:10}
	& \geq  M(\bm{\psi}_{n}^{j'}) + o_{n}(1), M(\bm{w}_{n}^{J})+ o_{n}(1), 
	\\ \notag
	\| \bm{\varphi}_{n}\|_{\dot{H}_{\gamma}^{1}}^{2}
	&=\sum_{j=1}^{J} \|\bm{\psi}_{n}^{j}\|_{\dot{H}_{\gamma}^{1}}^{2}
	+\|\bm{w}_{n}^{J}\|_{\dot{H}_{\gamma}^{1}}^{2}+ o_{n}(1)
	\\ \notag
	&\geq \|\bm{\psi}_{n}^{j'}\|_{\dot{H}_{\gamma}^{1}}^{2}+ o_{n}(1),
	\|\bm{w}_{n}^{J}\|_{\dot{H}_{\gamma}^{1}}^{2}+ o_{n}(1)
\end{align} 
Now, by $\bm{\varphi}_{n} \in PW_{\gamma}^{+}$ and Proposition \ref{prop2.6.0} and Corollary \ref{Global well-posedness}, we have $\|Q\|_{L^2(\mathbb{R})}^{1-s_c}\|\partial_xQ\|_{L^2(\mathbb{R})}^{s_c} -\delta
\geq \|\bm{\varphi}_{n}\|_{L^2(\mathcal{G})}^{1-s_c} \| \bm{\varphi}_{n}\|_{\dot{H}_{\gamma}^{1}(\mathcal{G})}^{s_{c}}$ for some constant $\delta>0$. Therefore, there exists $\widetilde{\delta}>0$ such that, for sufficiently large $n$, we have
\begin{align}
\label{eq:13}
	\|Q\|_{L^2(\mathbb{R})}^{1-s_c}\|\partial_xQ\|_{L^2(\mathbb{R})}^{s_c} -\widetilde{\delta}
	\geq \|\bm{\psi}_{n}^{j}\|_{L^2(\mathcal{G})}^{1-s_c} \| \bm{\psi}_{n}^{j}\|_{\dot{H}_{\gamma}^{1}(\mathcal{G})}^{s_{c}},\ 
	 \|\bm{w}_{n}^{J}\|_{L^2(\mathcal{G})}^{1-s_c} \|\bm{w}_{n}^{J}\|_{\dot{H}_{\gamma}^{1}(\mathcal{G})}^{s_{c}}, 
\end{align}
for all $1\leq j \leq J$. 
These inequalities give $E_{\gamma}(\bm{\psi}_{n}^{j}), E_{\gamma}(\bm{w}_{n}^{J}) \geq 0$ (see Remark \ref{rmk2.7}). 
Thus we get
\begin{align}
\notag
	E_{\gamma}( \bm{\varphi}_{n}) &= \sum_{j=1}^{J} E_{\gamma}(\bm{\psi}_{n}^{j}) + E_{\gamma}(\bm{w}_{n}^{J}) +o_{n}(1)
	\\ \label{eq:11}
	&\geq  E_{\gamma}(\bm{\psi}_{n}^{j'}) +o_{n}(1), E_{\gamma}(\bm{w}_{n}^{J}) +o_{n}(1)
\end{align}
for all $1\leq j'\leq J$. Since we assume $\mathcal{ME}^{c} < M^\text{line}(Q)^{(1-s_c)/s_c}E_0^\text{line}(Q)$, it holds from \eqref{eq:10}, \eqref{eq:11}, and $M^\text{line}(Q)^{\frac{1-s_{c}}{s_{c}}}E_0^\text{line}(Q)
> M(\bm{\varphi}_{n})^{\frac{1-s_{c}}{s_{c}}}E_\gamma(\bm{\varphi}_{n})$ that
\begin{align}
\label{eq:12}
	M^\text{line}(Q)^{\frac{1-s_{c}}{s_{c}}}E_0^\text{line}(Q)
	\geq M(\bm{\psi}_{n}^{j})^{\frac{1-s_{c}}{s_{c}}}E_\gamma(\bm{\psi}_{n}^{j}),\ 
	M(\bm{w}_{n}^{J})^{\frac{1-s_{c}}{s_{c}}}E_\gamma(\bm{w}_{n}^{J})
\end{align}
for all $1\leq j \leq J$ and sufficiently large $n$. Therefore, we get $\bm{\psi}_{n}^{j}, \bm{w}_{n}^{J} \in PW_{\gamma}^{+}$ for sufficiently large $n$ and all $ j\in \{1,2,\cdots,J\}$ by \eqref{eq:13} and \eqref{eq:12}. 

Let $J^{*} \in \{1,2,\cdots, \infty\}$ be the minimum number such that $\bm{w}_{n}^{J} =0$ for all $J > J^{*}$. 
We will show $J^{*}=1$ by contradiction. 

First of all, suppose that $J^{*}=0$. By Proposition \ref{Prop:LPD}, if $J^{*}=0$, then $\|e^{it \Delta_{\mathcal{G}}^{\gamma}} \bm{\varphi}_{n}\|_{L^a(\mathbb{R}; L^r)} \to 0$ as $n \to \infty$. Then, by Proposition \ref{LTP}, we get $\| \bm{u}_{n}\|_{L^a((0,\infty); L^r)} \lesssim 1$ for large $n \in \mathbb{N}$. This contradicts the definition of $\bm{u}_{n}$. 

Secondly, suppose that $J^{*} \geq 2$. Then there exists $\delta>0$ such that $M(\bm{\psi}_n^j)^{(1-s_c)/s_c}E_{\gamma}(\bm{\psi}_n^j) <\mathcal{ME}^{c} -\delta$ for all $1\leq j \leq J^{*}$. 
Reordering the profiles, we may choose $J_0, \cdots, J_4$ such that 
\begin{align*}
	J_{0}:=1 \leq j \leq J_1 
	&\Rightarrow y_n^j =0 \ (\forall n \in \mathbb{N})
	\text{ and } t_n^j =0\ (\forall n \in \mathbb{N}) ,
	\\
	J_1+1 \leq j \leq J_2 
	&\Rightarrow y_n^j =0 \ (\forall n \in \mathbb{N})
	\text{ and } |t_n^j|\to \infty \ (n \to \infty) ,
	\\
	J_2+1 \leq j \leq J_3 
	&\Rightarrow y_n^j \to \infty \ (n \to \infty)
	\text{ and } t_n^j =0\ (\forall n \in \mathbb{N}) ,
	\\
	J_3+1 \leq j \leq J_4 
	&\Rightarrow y_n^j \to \infty \ (n \to \infty)
	\text{ and } |t_n^j|\to \infty \ (n \to \infty) ,
\end{align*}
where we regard that there is no $j$ such that $a \leq j \leq b$ if $a>b$. 

We consider the case $J_{0} = 1 \leq j \leq J_{1}$. It is obvious from Proposition \ref{Prop:LPD} that $J_{1}=0$ or $1$. We may only consider the case $J_1=1$. In this case, $\bm{\psi}_{n}^1 = \mathcal{T}_{0,0} (\psi_{1,j}, \psi_{2,j}, \psi_{3,j})$ is independent of $n$ and so we write it by $\bm{\psi}^1$. We have $0<M(\bm{\psi}^{1})^{(1-s_c)/s_c}E_{\gamma}(\bm{\psi}^{1})  <\mathcal{ME}^c-\delta$ by $(t_n^j,y_n^j)=(0,0)$ and the assumption of $J^{*} \geq 2$. Hence, by the definition of $\mathcal{ME}^c$, we can find $\bm{U}^{1} \in C(\mathbb{R};H_{c}^1(\mathcal{G})) \cap L_t^{a}(\mathbb{R}; L_x^r(\mathcal{G}))$ such that 
\begin{align*}
	\bm{U}^{1}(t,x)= e^{it \Delta_{\mathcal{G}}^{\gamma} } \bm{\psi}^1 - i \int_{0}^{t} e^{i(t-s) \Delta_{\mathcal{G}}^{\gamma}} \bm{\mathcal{N}}(\bm{U}^{1})ds. 
\end{align*}
We write $\bm{U}_{n}^{1}=\bm{U}^{1}$ even though it does not depend on $n$. 

We consider the case of $J_1+1 \leq j \leq J_2$. In this case, we have $\bm{\psi}_{n}^{j} = \mathcal{T}_{t_{n}^{j},0} (\psi_{1,j}, \psi_{2,j}, \psi_{3,j}) = e^{-it_{n}^{j}\Delta_{\mathcal{G}}^{\gamma}} ( \Lambda_{1,0}  \psi_{1,j}+ \Lambda_{2,0} \psi_{2,j}+ \Lambda_{3,0}  \psi_{3,j})$. For simplicity, we write $\bm{\psi}^{j} = ( \Lambda_{1,0}  \psi_{1,j}+\Lambda_{2,0} \psi_{2,j}+\Lambda_{3,0}  \psi_{3,j})$. Now, we have $E_{\gamma}(\bm{\psi}_{n}^{j})+\frac{\omega}{2}M(\bm{\psi}_{n}^{j}) < \mathfrak{n}_{\omega,\gamma} -\delta$ for some $\delta>0$ by Corollary \ref{cor2.5}. By the dispersive estimate, we get $L_{\gamma}(\bm{\psi}^{j})+\frac{\omega}{2}M(\bm{\psi}^{j}) < \mathfrak{n}_{\omega,\gamma}$ since $|t_n^j| \to \infty$. 
If $j$ satisfies $t_n^j \to -\infty$, then we define a solution $\bm{U}^j$ to \eqref{NLS} that scatters to $\bm{\psi}^j$ as $t \to + \infty$ by Proposition \ref{WOp}. Then we have $\bm{U}^j \in C((0,\infty);H_{c}^1(\mathcal{G})) \cap L_t^{a}((0,\infty); L_x^r(\mathcal{G}))$. Since $\bm{\psi}^j \in PW_{\gamma}^{+}$, the solution $\bm{U}^j$ to \eqref{NLS} exists globally in both time directions. We will claim that $\bm{U}^j \in L_t^{a}(\mathbb{R}; L_x^r(\mathcal{G}))$. 
Since we have $\bm{\psi}^j \in PW_{\gamma}^{+}$, $\bm{\psi}^j$ satisfies that $M(\bm{U}^{j})^{(1-s_c)/s_c}E_{\gamma}(\bm{U}^{j})  = \lim_{n \to \infty} M(e^{it_{n}^{j}\Delta_{\mathcal{G}}^{\gamma}} \bm{\psi}^j)^{(1-s_c)/s_c}E_{\gamma}(e^{it_{n}^{j}\Delta_{\mathcal{G}}^{\gamma}} \bm{\psi}^j)  \leq \mathcal{ME}^c-\delta$. Therefore,  by the definition of $\mathcal{ME}^c$, we obtain $\bm{U}^j \in C(\mathbb{R};H_{c}^1(\mathcal{G})) \cap L_t^{a}(\mathbb{R}; L_x^r(\mathcal{G}))$. We set $\bm{U}_{n}^j (t,x):=\bm{U}^j (t-t_n^j,x)$, which obviously belongs to $L_t^{a}(\mathbb{R}; L_x^r(\mathcal{G}))$. If $j$ satisfies $t_n^j \to \infty$, then we can define $\bm{U}^j$ and  $\bm{U}_{n}^j (t,x):=\bm{U}^j (t-t_n^j,x)$ in the similar way.

We consider the case of $J_2+1 \leq j \leq J_3$. In this case, we have $\bm{\psi}_{n}^{j}=\mathcal{T}_{0,y_{n}^{j}}(\psi_{1,j},\psi_{2,j},\psi_{3,j})=\Lambda_{1,y_{n}^{j}} \psi_{1,j}+\Lambda_{2,y_{n}^{j}} \psi_{2,j}+\Lambda_{3,y_{n}^{j}} \psi_{3,j}$. For $k=1,2,3$, let $u_{k,j}$ be the solution to NLS on $\mathbb{R}$ with the initial data $\psi_{k,j}$, that is
\begin{align*}
	u_{k,j} = e^{it \partial_{xx} } \psi_{k,j} + i \int_{0}^{t} e^{i(t-s) \partial_{xx} } |u_{k,j}|^{p-1}u_{k,j} ds. 
\end{align*}
By the decoupling of profiles, i.e., Proposition \ref{prop5.1}, we find that $\psi_{k,j}$ belongs to $PW_{0}^{+}(\mathbb{R})$ which is the potential well set with the positive virial functional in the line case. By e.g. \cite{AkaNaw13}, we see that $u_{k,j}$ is global and $u_{k,j} \in  L_t^{a}(\mathbb{R}; L_x^r(\mathbb{R}))$. Therefore, by Proposition \ref{Prop:NPD2}, we have
\begin{align*}
	\bm{U}_{n}^{j} = e^{it\Delta_\mathcal{G}^{\gamma}} \bm{\psi}_{n}^{j} 
	- i \int_{0}^{t} e^{i(t-s)\Delta_\mathcal{G}^{\gamma}} \bm{\mathcal{N}}(\bm{U}_{n}^{j}) ds + \bm{R}_{n}
\end{align*}
with 
\begin{align*}
	\bm{U}_{n}^{j} =\mathcal{T}_{0,y_{n}^{j}}( u_{1,j}, u_{2,j}, u_{3,j} ) 
	\text{ and } 
	\|\bm{R}_{n}\|_{L_{t}^{a}L_{x}^{r}} \to 0 \text{ as } n \to \infty.
\end{align*}
Since $u_{k,j} \in  L_t^{a}(\mathbb{R}; L_x^r(\mathbb{R}))$ for $k=1,2,3$, we have $\bm{U}_{n}^{j} \in  L_t^{a}(\mathbb{R}; L_x^r(\mathcal{G}))$.

We consider the case of $J_3+1 \leq j \leq J_4$. In this case, we have 
\begin{align*}
	\bm{\psi}_{n}^{j} &= \mathcal{T}_{t_{n}^{j},y_{n}^{j}} (\psi_{1,j}, \psi_{2,j}, \psi_{3,j}) 
	\\
	&= e^{-it_{n}^{j}\Delta^{\gamma}_\boG} ( \Lambda_{1,y_{n}^{j}} \psi_{1,j}+ \Lambda_{2,y_{n}^{j}} \psi_{2,j}+ \Lambda_{3,y_{n}^{j}} \psi_{3,j})
	=:e^{-it_{n}^{j}\Delta^{\gamma}_\boG} \bm{\widetilde{\psi}}_{n}^{j}.
\end{align*}
Now, we have $E_{\gamma}(\bm{\psi}_{n}^{j})+\frac{\omega}{2}M(\bm{\psi}_{n}^{j}) < \mathfrak{n}_{\omega,\gamma} -\delta$ for some $\delta>0$ by Corollary \ref{cor2.5}. By the dispersive estimate, we get $L_{\gamma}(\bm{\widetilde{\psi}}_{n}^{j})+\frac{\omega}{2}M(\bm{\widetilde{\psi}}_{n}^{j}) < \mathfrak{n}_{\omega,\gamma}$ for large $n$ since $|t_n^j| \to \infty$. 
By Proposition \ref{prop5.1}, we have $\frac{1}{2} \|\partial_{x} \psi_{k,j}\|_{L^{2}}^{2} + \frac{\omega}{2} \| \psi_{k,j}\|_{L^{2}}^{2} < \mathfrak{n}_{\omega,\gamma}$ for $k=1,2,3$. Thus, if $j$ satisfies $t_n^j \to -\infty$, we can define the solution $u_{k,j}$ ($k=1,2,3$) to \eqref{NLSL} with $\|u_{k,j} - e^{it\partial_{xx}} \psi_{k,j}\|_{H^{1}(\mathbb{R})}\to 0$ as $t \to \infty$. Now, we have $M^{\mathrm{line}}(u_{k,j})^{\frac{1-s_{c}}{s_{c}}}E_0^\mathrm{line}(u_{k,j}) < M^\mathrm{line}(Q)^{\frac{1-s_{c}}{s_{c}}} E_0^\mathrm{line}(Q)$ and $\| u_{k,j}\|_{L^{2}(\mathbb{R})}^{1-s_c}\| \partial_{x}u_{k,j} \|_{L^{2}(\mathbb{R})}^{s_c} < \|Q\|_{L^{2}(\mathbb{R})}^{1-s_c}\|Q\|_{\dot{H}^{1}(\mathbb{R})}^{s_c}$. 
By e.g. \cite{AkaNaw13}, we see that $u_{k,j}$ is global and $u_{k,j} \in  L_t^{a}(\mathbb{R}; L_x^r(\mathbb{R}))$. Therefore, by Proposition \ref{Prop:NPD3}, we have
\begin{align*}
	\bm{V}_{n}^{j} = e^{it\Delta_{\mathcal{G}}^{\gamma}} \bm{\psi}_{n}^{j} 
	- i \int_{0}^{t} e^{i(t-s)\Delta_\mathcal{G}^{\gamma}} \bm{\mathcal{N}}(\bm{V}_{n}^{j}) ds + \bm{R}_{n}
\end{align*}
with 
\begin{align*}
	\bm{V}_{n}^{j}=\mathcal{T}_{0,y_{n}^{j}}( u_{1,j}, u_{2,j}, u_{3,j} )(t-t_{n}^{j})
	\text{ and } 
	\|\bm{R}_{n}\|_{L_{t}^{a}L_{x}^{r}} \to 0 \text{ as } n \to \infty.
\end{align*}
Since $u_{k,j} \in  L_t^{a}(\mathbb{R}; L_x^r(\mathbb{R}))$ for $k=1,2,3$, we have $\bm{V}_{n}^{j} \in  L_t^{a}(\mathbb{R}; L_x^r(\mathcal{G}))$. Define $\bm{U}_{n}^{j}(t,x):=\bm{V}_{n}^{j}(t,x)$. 
If $t_n^j \to \infty$, we can define $\bm{U}_{n}^{j}$ in the similar way. 
%

It is obvious that the profiles $\bm{U}_{n}^{j}$ can be written by $\mathcal{T}_{0,y_{n}^{j}}(u_{1}^{j},u_{2}^{j},u_{3}^{j})(t-t_{n}^{j})$ for some $u_{k}^{j} \in L_t^{a}(\mathbb{R}; L_x^r(\mathbb{R}))$. 
We define nonlinear profile by 
\begin{align*}
	\bm{Z}_{n}^{J} := \sum_{j=1}^{J} \bm{U}_{n}^{j}.
\end{align*}
By propositions for profiles, we obtain
\begin{align*}
	\bm{Z}_{n}^{J} = e^{it\Delta_{\mathcal{G}}^{\gamma}} (\bm{\varphi}_{n}- \bm{w}_{n}^{J}) - i\bm{z}_{n}^{J} + \bm{r}_{n}^{J},
\end{align*}
where $\|\bm{r}_{n}^{J}\|_{L_{t}^{a}L_{x}^{r}} \to 0$ as $n \to \infty$ and 
\begin{align*}
	\bm{z}_{n}^{J}=\sum_{j=1}^{J} \int_{0}^{t} e^{i(t-s)\Delta_{\mathcal{G}}^{\gamma}} \bm{\mathcal{N}}(\bm{U}_{n}^{j}) ds.
\end{align*}
We show
\begin{align*}
	\left\| \bm{z}_{n}^{J} - \int_{0}^{t} e^{i(t-s)\Delta_{\mathcal{G}}^{\gamma}}\mathcal{N}(\bm{Z}_{n}^{J}) ds \right\|_{L_{t}^{a}L_{x}^{r}} \to 0 
	\text{ as } n \to \infty. 
\end{align*}
By \eqref{Est:InnLq} and the Strichartz estimate, one can easily obtain
\begin{align*}
	\left\| \bm{z}_{n}^{J} - \int_{0}^{t} e^{i(t-s)\Delta_{\mathcal{G}}^{\gamma}} \bm{\mathcal{N}}(\bm{Z}_{n}^{J}) ds \right\|_{L_{t}^{a}L_{x}^{r}}
	&\lesssim 
	\left\| \bm{\mathcal{N}}\left(\sum_{j=1}^{J} \bm{U}_{n}^{j}\right) - \sum_{j=1}^{J} \bm{\mathcal{N}}(\bm{U}_{n}^{j}) \right\|_{L_{t}^{b'}L_{x}^{r'}}
	\\
	&\leq C_{J} \sum_{j\neq j'} \left\| |\bm{U}_{n}^{j}|^{p-1}|\bm{U}_{n}^{j'}| \right\|_{L_{t}^{b'}L_{x}^{r'}}.
\end{align*}
By Lemma \ref{lem4.4}, the right hand side goes to $0$ as $n\to \infty$. 
Thus we obtain
\begin{align*}
	\bm{Z}_{n}^{J} = e^{it\Delta^{\gamma}_\boG} (\bm{\varphi}_{n}- \bm{w}_{n}^{J}) 
	- i\int_{0}^{t} e^{i(t-s)\Delta^{\gamma}_\boG} \bm{\mathcal{N}}(\bm{Z}_{n}^{J}) ds 
	+ \bm{s}_{n}^{J},
\end{align*}
where $\|\bm{s}_{n}^{J}\|_{L_{t}^{a}L_{x}^{r}} \to 0$ as $n \to \infty$. 
We also have a bound of $\sup_{J} \limsup_{n \to \infty} \|\bm{Z}_{n}^{J}\|_{L_{t}^{a}L_{x}^{r}}$. Indeed, Lemma \ref{lem4.3} implies 
\begin{align*}
	\limsup_{n \to \infty} \|\bm{Z}_{n}^{J}\|_{L_{t}^{a}L_{x}^{r}}^{p}
	\lesssim \sum_{j=1}^{J} \sum_{k=1}^{3} \|u_{k}^{j}\|_{L_{t}^{a}L_{x}^{r}}^{p}.
\end{align*}

For simplicity, we set $a^j: = C\sum_{k=1}^{3} \|u_{k}^{j}\|_{L_{t}^{a}L_{x}^{r}}^{p}$. 
There exists a finite set $\mathcal{J}$ such that $\|\bm{\psi}^j\|_{H^1}< \varepsilon_0$ for any $j \not\in \mathcal{J}$, where $\varepsilon_0$ is the universal constant in the small data scattering result, Proposition \ref{Small data scattering}. By Proposition \ref{Small data scattering} and the pythagorean decompositions of $\dot{H}_{\gamma}^{1}$-norm and $L^2$-norm, we have
\begin{align*}
\limsup_{n\to \infty} \|\bm{Z}_{n}^{J}\|_{L_{t}^{a}L_{x}^{r}}^{p} 
& \leq  \sum_{j=1}^{J} a^j
\leq \sum_{j\in \mathcal{J}} a^j + \sum_{j \not\in \mathcal{J}}  a^j 
\\
& \lesssim  \sum_{j\in \mathcal{J}}  a^j 
+ \sum_{j \not\in \mathcal{J}}  
\|\bm{\psi}^j \|_{\dot{H}_{\gamma}^{1}}^{2}
\\
& \lesssim  \sum_{j\in \mathcal{J}}  a^j + \limsup_{n\to \infty}  \| \bm{\varphi}_n\|_{\dot{H}_{\gamma}^{1}}^{2}
\\
& \lesssim  \sum_{j\in \mathcal{J}}  a^j + C
\\
&\leq A,
\end{align*}
where $A$ is independent of $J$.

By Proposition \ref{Prop:LPD}, we can choose $J$ large enough in such a way that $\limsup_{n\to \infty} \|e^{it\Delta_{\mathcal{G}}^{\gamma}} \bm{w}_n^J\|_{L_t^a L_x^r} < \varepsilon$, where $\varepsilon =\varepsilon(A)>0$ is the number in Proposition \ref{LTP}. Then, we get $\|\bm{u}_n\|_{L_t^a L_x^r}<\infty$ for large $n$, and this contradicts $\|\bm{u}_n\|_{L_t^a L_x^r}=\infty$. 

Therefore, we obtain $J^{*}=1$. Namely we have
\begin{align*}
	\bm{\varphi}_{n} = \mathcal{T}_{t_{n}^{1},y_{n}^{1}} (\psi_{1,1}, \psi_{2,1}, \psi_{3,1}) + \bm{w}_{n}^{1}.
\end{align*}
We also have $y_{n}^{1}\equiv 0$. Indeed, if $y_{n}^{1} \to \infty$, then the  above argument implies a contradiction to non-scattering assumption, where we use the result on $\mathbb{R}$ by e.g. \cite{AkaNaw13}. 

Let $\bm{u}^c$ be a global solution of \eqref{NLS} with the initial data $\bm{\psi}^1$ if $t_{n}^{1}=0$ or the final data $\bm{\psi}^1$ if $|t_{n}^{1}| \to \infty$. 
We show $\| \bm{u}^c\|_{L^a(\mathbb{R};L^r)} = \infty$. To prove it, we suppose that $\|\bm{u}^c\|_{L^{a}(\mathbb{R};L^r)} < \infty$.
It follows that $\bm{\varphi}_n - \bm{u}^c(t_n^1)= e^{it_n^1 \Delta_{\mathcal{G}}^{\gamma}} \bm{\psi}^1 - \bm{u}^c(t_n^1)+\bm{w}_n^1$ and thus we have
\begin{align*}
	\|e^{it\Delta_{\mathcal{G}}^{\gamma}} \left( \bm{\varphi}_n - \bm{u}^c(t_n^1) \right)\|_{L^{a}(\mathbb{R};L^r)} \to 0 \text{ as } n \to \infty. 
\end{align*}
By Proposition \ref{LTP}, it holds that $\bm{u}_n \in L^a(\mathbb{R};L^r(\mathbb{R}))$ for sufficiently large $n$. This is a contradiction. Thus, we get $\|\bm{u}^c\|_{L^a(\mathbb{R};L^r)} = \infty$. Moreover, we have $M(\bm{u}^{c})^{(1-s_c)/s_c}E_{\gamma}(\bm{u}^{c}) =\mathcal{ME}^{c}$. Thus, we get a critical element $\bm{u}^c$.
\end{proof}

For a forward critical element, we have the following compactness property. We omit the case of backward critical element since they can be treated similarly. 

\begin{lemma}[Compactness of the critical element]
Let $\bm{u}$ be a forward critical element.
Then $\{\bm{u}(t):t\in [0,\infty) \}$ is precompact in $H_{c}^1(\mathcal{G})$.
\end{lemma}

The compactness can be rewritten as follows. 

\begin{lemma}
\label{lem6.3}
Let $\bm{u}$ be a solution to \eqref{NLS} satisfying that $\{\bm{u}(t):t\in [0,\infty) \}$ is precompact in $H_{c}^1(\mathcal{G})$. Then, for any $\varepsilon>0$, there exists $R=R(\varepsilon)>0$ such that
\begin{align}
\label{eq:cpt}
	\sum_{k=1}^{3} \int_{x>R}  |\partial_{x} u_{k}(t,x)|^2  + |u_{k}(t,x)|^2 + |u_{k}(t,x)|^{p+1} dx <\varepsilon
	\text{ for any } t\in [0,\infty).
\end{align}
\end{lemma}

These lemmas can be obtained  straightforwardly. See e.g. \cite{HoRo08,DHR08}.

\subsection{Extinction of the critical element}

We derive a contradiction by compactness of the critical element and the localized virial identity. 

\begin{lemma}[Rigidity]
\label{rigidity}
If $\bm{u}_{0} \in PW_{\gamma}^{+}$ and the corresponding solution $\bm{u}$ satisfies \eqref{eq:cpt}, then $\bm{u} = 0$.
\end{lemma}

The critical element must be zero by the above lemma. However, it contradicts that the mass-energy of the critical element is not zero. Therefore, the assumption $\mathcal{ME}^{c} < M^{\mathrm{line}}(Q)^{(1-s_{c})/s_{c}}E_{0}^{\mathrm{line}}(Q)$ is false and thus we get the statement.

\begin{proof}[Proof of Lemma \ref{rigidity}]
Using Proposition \ref{Generalized virial identities}, we have
\begin{align*}
	V''(t)= 4K_\gamma(\bm{u}(t)) + \mathcal{R}_1 + \mathcal{R}_2 + \mathcal{R}_3,
\end{align*}
where $\mathcal{R}_j=\mathcal{R}_j(t)$ $(j=1,2,3)$ are defined as
\begin{align*}
	\mathcal{R}_1
		&:= 4\int_\mathcal{G} \left\{\mathscr{X}''\left(\frac{x}{R}\right)-2\right\}|\partial_x \bm{u}(t,x)|^2dx, \\
	\mathcal{R}_2
		&:= -\frac{2(p-1)}{p+1}\int_\mathcal{G} \left\{\mathscr{X}''\left(\frac{x}{R}\right)-2 \right\}|\bm{u}(t,x)|^{p+1}dx,   \\
	\mathcal{R}_3
		&:= -\int_\mathcal{G} \left\{\frac{1}{R^2}\mathscr{X}^{(4)}\left(\frac{x}{R}\right)\right\}|\bm{u}(t,x)|^2dx.
\end{align*}
Let $\varepsilon>0$. By Lemma \ref{lem6.3}, we have
\begin{align*}
	|\mathcal{R}_1|+|\mathcal{R}_2|+|\mathcal{R}_3| 
	\lesssim \sum_{k=1}^{3} \int_{x>R}  |\partial_{x} u_{k}(t,x)|^2  + |u_{k}(t,x)|^2 + |u_{k}(t,x)|^{p+1} dx <\varepsilon
\end{align*}
for large $R>0$. 
Since the critical element belongs to $PW_{\gamma}^{+}$, there exists $\delta>0$ independent of $t$ such that $K_{\gamma}(\bm{u}(t))>\delta$ by Corollary \ref{Global well-posedness}. Therefore, by taking $\varepsilon>0$ sufficiently small, i.e., $R>0$ sufficiently large, then we get $V''(t) \gtrsim \delta$ for any $t \in (0,\infty)$. On the other hand, the mass conservation law shows us $V(t) \leq R^2 \|\bm{u}(t)\|_{L^2}^2<C$, where $C$ is independent of $t$, for any $t\in (0,\infty)$. Therefore, we reach a contradiction. 
\end{proof}

\appendix

\section{Symmetries on star graphs}
\label{appA}

Let us consider finite group $G$ acting on $H^1_c(\boG)$. 
(More specifically $G$ is a finite subgroup of the set $\boB (H^1_c(\boG))$ of bounded linear operators on $H^1_c(\boG)$.) 
We assume that $G$ satisfies the following conditions:
$$
\si\Dgg\bou = \Dgg \si\bou\qquad (\forall \bou\in \boD(\Dgg),\ \si\in G),
$$
$$
\si\boN_{\mu}(\bou) = \boN_{\mu} (\si\bou)\qquad (\forall 
\bou\in H^1_c(\boG)\ \ \si\in G).
$$
We recall $\boN_{\mu} (\bou) = \mu |\bou|^{p-1}\bou$. 
Then the following holds:
\begin{lem}
If $\bou$ is a solution to NLS:
\begin{equation}\label{e1}
i\rd_t \bou + \Dgg \bou = \boN_{\mu}(\bou),
\end{equation}
then so is $\si \bou$ for all $\si\in G$.
\end{lem}

We then define 
$$
\frak{G} := \{ \bou\in H^1_c(\boG)\ :\ \si \bou = \bou \quad (\forall \si\in G) \}.
$$
Note that $\frak{G}$ is a $G$-invariant closed subspace of $H^1_c(\boG)$. 

\begin{lem}
\eqref{e1} preserves $\frak{G}$; namely if $\bou(0)\in \frak{G}$, then $\bou(t)\in \frak{G}$ for all $t$.
\end{lem}

\begin{proof}
$\bou$ and $\si\bou$ satisfy \eqref{e1} with the same data, we have $\bou (t)= \si \bou(t)$ for all $t$ by uniqueness.
\end{proof}

\begin{exam}
Define $g_{23} \in \boB (H^1_c (\boG))$ by
$$
g_{23} 
\begin{pmatrix}
u_1\\ u_2 \\u_3
\end{pmatrix}
:=
\begin{pmatrix}
u_1\\ u_3 \\u_2
\end{pmatrix}
.
$$
Then $G = \jb{g_{23}}$ satisfies the conditions ($\# G=2$). In this case
$$
\frak{G} = \{ \bou\in H^1_c(\boG)\ :\ u_2=u_3\}.
$$
Note that the ground states constructed in Adami et al. are included in $\frak{G}$.
\end{exam}

\begin{exam}
Define $\tilde{g}_{23} \in \boB (H^1_c (\boG))$ by
$$
\tilde{g}_{23} 
\begin{pmatrix}
u_1\\ u_2 \\u_3
\end{pmatrix}
:=
-
\begin{pmatrix}
u_1\\ u_3 \\ u_2
\end{pmatrix}
.
$$
Then $G = \jb{ \tilde{g_{23}} }$ satisfies the conditions ($\# G=2$). In this case
$$
\frak{G} = \{ \bou\in H^1_c(\boG)\ :\ u_1=0,\ u_2=-u_3\}.
$$
This space is essentially $H^1_{\odd} (\R)$.
\end{exam}

\begin{exam}
Define $\si \in \boB (H^1_c (\boG))$ by
$$
\si 
\begin{pmatrix}
u_1\\ u_2 \\u_3
\end{pmatrix}
:=
\begin{pmatrix}
u_2\\ u_3 \\ u_1
\end{pmatrix}
.
$$
Then $G = \jb{ \si }$ satisfies the conditions ($\# G=3$). In this case
$$
\frak{G} = H_\mathrm{rad}^1(\mathcal{G}).
$$
\end{exam}

\begin{exam}
Let $\om$ denote the cubic root of $1$ (either of $\om= \frac{-1 \pm \sqrt{3}i}{2}$).
Define $\tilde{\si} \in \boB (H^1_c (\boG))$ by
$$
\tilde{\si} 
\begin{pmatrix}
u_1\\ u_2 \\u_3
\end{pmatrix}
:=
\om
\begin{pmatrix}
u_2\\ u_3 \\ u_1
\end{pmatrix}
.
$$
Then $G = \jb{ \tilde\si }$ satisfies the conditions ($\# G=3$). In this case
$$
\frak{G} = \{ \bou\in H^1_c(\boG)\ :\ u_1=\om u_2 =\om^2 u_3\}.
$$
It is fair to say that this is a natural generalization of odd function on $\boG_2$ ($=\R$).
\end{exam}



\begin{acknowledgements}
This work was supported by JSPS KAKENHI No. JP22J00787 (for M.H.), No. JP19K14581 (for M.I.), No. JP18K13444 and Overseas Research Fellowship (for T.I.), and No. JP19H05599 (for I.S.). 
\end{acknowledgements}



\noindent(M. Hamano): m.hamano3@kurenai.waseda.jp
\\
Faculty of Science and Engineering, Waseda University, Tokyo 169-8555, Japan.
\\
\vspace{10pt}
\\
\noindent(M. Ikeda): masahiro.ikeda@riken.jp
\\ Center for Advanced Intelligence Project, RIKEN, Tokyo 103-0027, Japan.
\\
Faculty of Science and Technology, Keio University, Yokohama 223-8522, Japan.
\\
\vspace{10pt}
\\
\noindent(T. Inui): inui@math.sci.osaka-u.ac.jp
\\Department of Mathematics, Graduate School of Science, Osaka University, Toyonaka, Osaka 560-0043, Japan.
\\ 
University of British Columbia, 1984 Mathematics Rd., Vancouver, British Columbia V6T1Z2, Canada.
\\
\vspace{10pt}
\\
\noindent(I. Shimizu): i.shimizu@sigmath.es.osaka-u.ac.jp
\\Mathematical Science, Graduate School of Engineering Science, Osaka University, Toyonaka, Osaka 560-0043, Japan.



\begin{thebibliography}{99}
\bibitem{AdaCacFinNoj14} 
R. Adami, C. Cacciapuoti, D. Finco, and D. Noja, 
\emph{Variational properties and orbital stability of standing waves for NLS equation on a star graph},
J. Differential Equations \textbf{257} (2014),  no. 10, 3738--3777. 

\bibitem{ACFN14}
R. Adami, C. Cacciapuoti, D. Finco,D. Noja, 
\emph{Constrained energy minimization and orbital stability for the NLS equation on a star graph}, 
Ann. Inst. H. Poincar\'e C Anal. Non Lin\'eaire \textbf{31} (2014), no. 6, 1289--1310.

\bibitem{AdaSerTil15} 
R. Adami, E. Serra, and P. Tilli, \emph{
NLS ground states on graphs
}, Calc. Var. Partial Differential Equations \textbf{54}(1) (2015), 743--761.


\bibitem{AdaCacFinNoj16} 
R. Adami, C. Cacciapuoti, D. Finco, and D. Noja, 
\emph{
Stable standing waves for a NLS on star graphs as local minimizers of the constrained energy
}, J. Differential Equations \textbf{260}(10) (2016), 7397--7415.


\bibitem{AdaSerTil16} 
R. Adami, E. Serra, and P. Tilli, \emph{
Threshold phenomena and existence results for NLS states on metric graphs
}, J. Funct. Anal. \textbf{271}(1) (2016), 201--223.


\bibitem{AkaNaw13}
T. Akahori, and H. Nawa, 
\emph{Blowup and scattering problems for the nonlinear Schr\"odinger equations},
Kyoto J. Math. {\bf 53} (2013), no. 3, 629--672. 

\bibitem{AngGol18} 
P. Angulo and J. Goloshchapova, 
\emph{Extension theory approach in the stability of the standing waves for the NLS equation with point interactions on a star graph}, 
Adv. Differential Equations \textbf{23} (2018), no. 11--12, 793--846. 


\bibitem{AngGol18Ins} J. Anglo Pava and N. Goloshchapova, 
\emph{
On the orbital instability of excited states for the NLS equation with the $\del$-interaction on a star graph,
} Discrete Contin. Dyn. Syst. \textbf{38}(10) (2018), 5039--5066.


\bibitem{AIMMU20} 
K. Aoki, T. Inui, H. Miyazaki, H. Mizutani, and K. Uriya, \emph{
Asymptotic behavior for the long-range nonlinear Schr\"odinger equation on star graph with the Kirchhoff boundary condition,
} Pure Appl. Anal. \textbf{4} (2022), no. 2, 287--311.


\bibitem{AokInuMiz21} 
K. Aoki, T. Inui, and H. Mizutani, \emph{
Failure of scattering to standing waves for a Schr\"odinger equation with long-range nonlinearity on star graph,
} J. Evol. Equ. \textbf{21}(1) (2021), 297--312.



\bibitem{BanIgn11} 
V. Banica and L. I. Ignat, \emph{
Dispersion for the Schr\"odinger equation on networks,
} J. Math. Phys. \textbf{52}(8) (2011), 083703.


\bibitem{BanIgn14} 
V. Banica and L. I. Ignat, \emph{Dispersion for the Schr\"odinger equation on the line with multiple Dirac delta potentials and on delta trees}, Anal. PDE \textbf{7}(4) (2014), 903--927.


\bibitem{BanVis16}
V. Banica and N. Visciglia, 
\emph{Scattering for NLS with a delta potential}, 
J. Differential Equations \textbf{260} (2016), no. 5, 4410--4439.


\bibitem{BerKuc13}
G. Berkolaiko and P. Kuchment, 
Introduction to Quantum Graphs, 
American Math. Society, Providence, 2013.



\bibitem{BerMerPel21} 
G. Berkolaiko, J. L. Marzuola, and D. E. Pelinovsky, \emph{
Edge-localized states on quantum graphs in the limit of large mass,
} Ann. Inst. H. Poincar\'e C Anal. Non Lin\'eaire \textbf{38}(5) (2021), 1295--1335.



\bibitem{BlaExnHav08}
J. Blank, P. Exner, and M. Havl\'i\v{c}ek, 
\emph{Hilbert Space Operators in Quantum Physics: Second Edition}, 
Springer, New York, 2008.


\bibitem{CacFinNoj15} 
C. Cacciapuoti, D. Finco, and D. Noja, 
\emph{Topology-induced bifurcations for the nonlinear Schr\"odinger equation on the tadpole graph,
} Phys. Rev. E (3) \textbf{91}(1) (2015), 013206.


\bibitem{CacFinNoj17}
C. Cacciapuoti, D. Finco, and D. Noja, 
\emph{Ground state and orbital stability for the NLS equation on a general starlike graph with potentials}, 
Nonlinearity \textbf{30}, (2017), no. 8, 3271--3303


\bibitem{CacDovSer18} 
C. Cacciapuoti, S. Dovetta, and E. Serra, \emph{
Variational and stability properties of constant solutions to the NLS equation on compact metric graphs,
} Milan J. Math. \textbf{86}(2) (2018), 305--327.


\bibitem{Caz03} 
T. Cazenave, 
\textit{Semilinear Schr\"{o}dinger Equations}, 
Courant Lecture Notes in Mathematics, vol. 10, American Mathematical Society, Courant Institute of Mathematical Sciences, 2003.

\bibitem{CazHar98}
T. Cazenave and A. Haraux, 
An introduction to semilinear evolution equations, 
The Clarendon Press, New York, 1998.



\bibitem{Dov18} 
S. Dovetta, \emph{
Existence of infinitely many stationary solutions of the $L^2$-subcritical and critical NLSE on compact metric graphs,
} J. Differential Equations \textbf{264}(7) (2018), 4806--4821.


\bibitem{Dov19} 
S. Dovetta, \emph{
Mass-constrained ground states of the stationary NLSE on periodic metic graphs,
} NoDEA Nonlinear Differential Equations Appl. \textbf{26}(5) (2019), Paper No. 30.


\bibitem{DovGhiMicPis20} 
S. Dovetta, M. Ghimenti, A. M. Micheletti, and A. Pistoia, 
\emph{
Peaked and low action solutions of NLS equations on graphs with terminal edges,
} SIAM J. Math. Anal. \textbf{52}(3) (2020), 2874--2894.


\bibitem{DovSerTil20} 
S. Dovetta, E. Serra, and P. Tilli, \emph{
NLS ground states on metric trees: existence results and open questions,
} J. Lond. Math. Soc. (2) \textbf{102}(3) (2020), 1223--1240.

\bibitem{DHR08} T. Duyckaerts, J. Holmer, S. Roudenko, 
\textit{Scattering for the non-radial 3D cubic nonlinear Schr\"{o}dinger equation}, 
Math. Res. Lett. {\bf 15} (2008), no. 6, 1233--1250.


\bibitem{FanXieCaz11}
D. Fang, J. Xie, and T. Cazenave, 
\emph{Scattering for the focusing energy-subcritical nonlinear Schr\"odinger equation},
Sci. China Math. {\bf 54} (2011), no. 10, 2037--2062. 

\bibitem{FarPas17}
L. G. Farah and A. Pastor, 
\emph{Scattering for a 3D coupled nonlinear Schr\"{o}dinger system}, 
J. Math. Phys. {\bf 58} (2017), no. 7, 071502, 33 pp.

\bibitem{Fos05} 
D. Foschi, 
\emph{Inhomogeneous Strichartz estimates}, 
J. Hyperbolic Differ. Equ. \textbf{2} (2005), no. 1, 1--24.

\bibitem{FukJea08} 
R. Fukuizumi and L. Jeanjean, 
\emph{Stability of standing waves for a nonlinear Schr\"odinger equation with a repulsive Dirac delta potential}, 
Discrete Contin. Dyn. Syst. \textbf{21} (2008), no. 1, 121--136. 

\bibitem{FukOhtOza08} 
R. Fukuizumi, M. Ohta, and T. Ozawa, 
\emph{Nonlinear Schr\"odinger equation with a point defect},
Ann. Inst. H. Poincar\'e Anal. Non Lin\'eaire \textbf{25} (2008), no. 5, 837--845. 

\bibitem{GnuSmiDer11}
S. Gnutzmann, U. Smilansky, and S. Derevyanko, 
Stationary scattering from a nonlinear network, 
Phys. Rev. A \textbf{83} (2011), 033831.


\bibitem{Gol19} 
N. Goloshchapova, 
\emph{A nonlinear Klein-Gordon equation on a star graph}, 
Math. Nachr. {\bf 294} (2021), no. 9, 1742--1764. 

\bibitem{GolOht20} 
N. Goloshchapova and M. Ohta, 
\emph{Blow-up and strong instability of standing waves for the NLS-$\delta$ equation on a star graph},
Nonlinear Anal. \textbf{196} (2020), 111753. 

\bibitem{GreIgn19} 
A. Grecu and L. Ignat, 
\emph{The Schr\"odinger equation on a star-shaped graph under general coupling conditions},
J. Phys. A \textbf{52} (2019), no. 3, 035202, 26 pp. 

\bibitem{Gue}
C. D. Guevara, 
\emph{Global behavior of finite energy solutions to the d-dimensional focusing nonlinear Schr\"{o}dinger equation}, 
Appl. Math. Res. Express. AMRX 2014, no. 2, 177--243.

\bibitem{Guo}
Q. Guo, 
\emph{Divergent solutions to the $L^2$-supercritical NLS equations}, 
Acta Math. Appl. Sin. Engl. Ser. \textbf{32} (2016), no. 1, 137--162.

\bibitem{HajStu04} 
H. Hajaiej and C. A. Stuart, 
\emph{On the variational approach to the stability of standing waves for the nonlinear Schr\"odinger equation},
Adv. Nonlinear Stud. \textbf{4} (2004), no. 4, 469--501. 


\bibitem{Ham18pre}
M. Hamano, 
\emph{Global dynamics below the ground state for the quadratic Sch\"odinger system in 5d}, 
preprint, arXiv:1805.12245

\bibitem{HoRo08} J. Holmer, S. Roudenko,
\textit{A sharp condition for scattering of the radial 3D cubic nonlinear Schr\"{o}dinger equation}, Comm. Math. Phys. {\bf 282} (2008), no. 2, 435--467. 

\bibitem{IkeInu17} 
M. Ikeda and T. Inui, 
\emph{Global dynamics below the standing waves for the focusing semilinear Schr\"odinger equation with a repulsive Dirac delta potential},
Anal. PDE \textbf{10} (2017), no. 2, 481--512. 
MR3619878

\bibitem{Inu17}
T. Inui, 
\emph{Global dynamics of solutions with group invariance for the nonlinear Schr\"odinger equation}, 
Commun. Pure Appl. Anal. \textbf{16} (2017), no. 2, 557--590.

\bibitem{InKiNi19}
T. Inui, N. Kishimoto, K. Nishimura,
\emph{Scattering for a mass critical NLS system below the ground state with and without mass-resonance condition}, 
Discrete Contin. Dyn. Syst. \textbf{39} (2019), no. 11, 6299--6353.

\bibitem{Kai19} 
A. Kairzhan, 
\emph{
Orbital instability of standing waves for NLS equation on star graphs,
} Proc. Amer. Math. Soc. \textbf{147}(7) (2019), 2911--2924.


\bibitem{KaiPel18HS} 
A. Kairzhan and D. E. Pelinovsky, 
\emph{Nonlinear instability of half-solitons on star graphs,
} J. Differential Equations \textbf{264}(12) (2018), 7357--7383.


\bibitem{KaiPel18SS} 
A. Kairzhan and D. E. Pelinovsky, 
\emph{
Spectral stability of shifted states on star graphs, 
} J. Phys. A \textbf{51}(9) (2018), 095203.



\bibitem{KenMer06}
C. E. Kenig, F. Merle,
\emph{Global well-posedness, scattering and blow-up for the energy-critical, focusing, non-linear Schr\"odinger equation in the radial case}, 
Invent. Math. \textbf{166} (2006), no. 3, 645--675.
MR2257393 

\bibitem{KosSch06}
V. Kostrykin and R. Schrader, 
\emph{Laplacians on metric graphs: eigenvalues, resolvents and semigroups}, 
Contemp. Math. \textbf{415} (2006), 201--225.


\bibitem{Miz20}
H. Mizutani, 
\emph{Wave operators on Sobolev spaces}, 
Proc. Amer. Math. Soc. \textbf{148} (2020), no. 4, 1645--1652. 

\bibitem{Nak99}
K. Nakanishi, 
\emph{Energy scattering for nonlinear Klein-Gordon and Schr\"odinger equations in spatial dimensions $1$ and $2$}, 
J. Funct. Anal. \textbf{169} (1999), no. 1, 201--225. 

\bibitem{NojPelSha15} 
D. Noja, D. Pelinovsky, and G. Shaikhova, 
\emph{Bifurcations and stability of standing waves in the nonlinear Schr\"odinger equation on the tadpole graph
}, Nonlinearity \textbf{28}(7) (2015), 2343--2378. 


\bibitem{NojPel20} 
D. Noja and D. E. Pelinovsky, 
\emph{Standing waves of the quintic NLS equation on the tadpole graph,
} Calc. Var. Partial Differential Equations \textbf{59}(5) (2020), Paper No. 173.



\bibitem{Pan18} 
A. Pankov, \emph{
Nonlinear Schr\"odinger equations on periodic metric graphs,
} Discrete Contin. Dyn. Syst. \textbf{38}(2) (2018), 697--714.


\bibitem{PelSch17} 
D. Pelinovsky and G. Schneider, \emph{
Bifurcation of standing localized waves on periodic graphs,
} Ann Henri Poincar\'e \textbf{18}(4) (2017), 1185--1211.


\bibitem{SerTen16} 
E. Serra and L. Tentarelli, \emph{
On the lack of bound states for certain NLS equations on metric graphs,
} Nonlinear Anal. \textbf{145} (2016), 68--82.


\bibitem{SMSSN10}
Z. Sobirov, D. Matrasulov, K. Sabirov, S. Sawada, and K. Nakamura, 
\emph{Integrable nonlinear Schr\"odinger equation on simple networks: Connection formula at vertices}, Phys. Rev. E \textbf{81} (2010), 066602.


































\end{thebibliography}
\end{document}